\documentclass[11pt]{amsart}

\usepackage{amssymb}
\usepackage{amsbsy}
\usepackage{latexsym}
\usepackage{amsmath}
\usepackage{amsthm}
\usepackage{amsfonts}
\usepackage{color}
\usepackage{hyperref}
\usepackage{pdfsync}
\usepackage[usenames,dvipsnames]{pstricks}
\usepackage{epsfig}
\usepackage{pst-grad} 
\usepackage{pst-plot} 
\newcommand{\1}{\mathbf{1}}
\newcommand{\R}{\mathbb{R}}

\newcommand{\Z}{\mathbb{Z}}

\newcommand{\VV}{\mathbf{V}}
\newcommand{\EE}{\mathbf{E}}
\newcommand{\osc}{\mathrm{osc}}

\newcommand{\VVi}{\mathbf{V}_\mathbf{i}}
\newcommand{\VVb}{\mathbf{V}_\mathbf{b}}

\newcommand{\e}{\varepsilon}
\newcommand{\ver}{\mathrm{\mathbf{v}}}

\theoremstyle{plain}
\newtheorem{defi}{Definition}[section]
\newtheorem{prop}[defi]{Proposition}
\newtheorem{teo}[defi]{Theorem}

\newtheorem{lema}[defi]{Lemma}

\theoremstyle{definition}
\newtheorem{rema}[defi]{Remark}

\theoremstyle{remark}

\numberwithin{equation}{section}

\def\ue{u^\e}
\def\varphib{\pmb{\varphi}}
\def\psib{\pmb{\psi}}
\def\chib{\pmb{\chi}}

\begin{document}

\title[]{Degenerate elliptic PDEs on a network with Kirchhoff conditions}

\author[]{Guy Barles}
\address{
Guy Barles: Institut Denis Poisson (UMR CNRS 7013)
Université de Tours, Université d’Orléans, CNRS.
Parc de Grandmont. 37200 Tours, France.
\\
{\tt guy.barles@idpoisson.fr}
}

\author[]{Olivier Ley}
\address{
Olivier Ley: Univ Rennes, INSA Rennes, CNRS, IRMAR - UMR 6625, F-35000 Rennes, France.
{\tt olivier.ley@insa-rennes.fr}
}

\author[]{Erwin Topp}
\address{
Erwin Topp:
Instituto de Matem\'aticas, Universidade Federal do Rio de Janeiro, Rio de Janeiro - RJ, 21941-909, Brazil. {\tt etopp@im.ufrj.br}
}

\date{\today}

\begin{abstract}
In this article, we are interested in semilinear, possibly degenerate elliptic equations posed on a general network, with nonlinear Kirchhoff-type conditions for its interior vertices and Dirichlet boundary conditions for the boundary ones. The novelty here is the generality of the equations posed on each edge that is incident to a particular vertex, ranging from first-order equations to uniformly elliptic ones.
Our main result is a strong comparison principle, i.e., a comparison result between discontinuous viscosity sub and supersolutions of such problems, from which we conclude  the existence and uniqueness of a continuous viscosity by Perron's method. Further extensions are also discussed.

\end{abstract}



\maketitle

\section{Introduction.}

A general network $\Gamma$ in $\R^d$ is made of a finite number of vertices $\ver\in\VV$ connected
with a finite number of edges $E\in \EE$. Each edge has the differential structure of a line segment and, after a suitable
parametrization, we can put a differential equation on each of them. It is desirable  that these equations are
edge-by-edge unrelated, in principle, if we look for potential applications, but also by natural mathematical questions. By imposing conditions on the vertices, we couple these equations into a system that may be regarded as a problem in the full network.
In this paper, we consider Kirchhoff-type conditions on some \textsl{interior vertices}
$\ver \in \VVi$
and suitable boundary conditions on the
remaining \textsl{boundary vertices} $\ver \in \VVb$,


To simplify the exposure in the introduction, we will only describe the simplified
case of a {\em junction}, which contains the main relevant difficulties,
and refer the reader to Section~\ref{sec:network} for the general case.
A junction is
a star-shaped network $\Gamma$ embedded in $\R^d$,
$d \geq 2$, where $O = (0,...,0) \in \R^d$ is the unique interior vertex, that is $\VVi =\{O\}$,
and we refer to it as the {\it junction} point.
We consider a family $\VVb = \{\ver_1, ..., \ver_N \}$ of $N$ boundary vertices, and $N$ (open) edges $\{ E_i \}_{1\leq i\leq N}$ given
by
$$
E_i=\{t\ver_i, t \in (0,1)\}.
$$

Of course, we assume that, for any $i,j$, $\ver_i\neq O$ and $\ver_i$, $\ver_j$ are not collinear; as a consequence the $E_i$'s are non-empty and, if $i\neq j$, $E_i\cap E_j =\emptyset$. With these notations, we have
\begin{equation*}
\Gamma :=  \bigcup_{i=1}^N \bar E_i,
\end{equation*}
where $\bar E_i$ denotes the closure of the set $E_i$ in $\R^d$.

For $x\in E_i$, we set $x_i:=|x|$ where $|\cdot |$ denotes the usual Euclidean norm in $\R^d$. We have in that way a natural parametrization by arc length of the (open) segment $E_i$ and its closure $\bar E_i$.
For a function $u:\Gamma\to\R$,
we define
$u_i : [0, \ell_i]\to \R$, where $\ell_i=|\ver_i|$, by
$$ u_i (x_i):=u(x)\quad \hbox{if $x\in \bar E_i$,}$$
and we use the notations $u_x(x)=u_{x_i}(x)=u'_i(x_i)$ and $u_{x x}(x)=u_{x_i x_i}(x)=u''_i(x_i)$ for  $x\in E_i$. 

With these notations, we can define in a proper way
our problem
\begin{eqnarray}\label{pde-junct}
&& \left\{\begin{array}{ll}
\lambda u (x) - a_i(x) u_{x x} (x) + H_i(x,u_{x})=0, & x\in E_i, \ 1\leq i\leq N, \\[2mm]
F(u(O), u_{x_1}(O), \cdots ,  u_{x_N}(O))=0, &\\[2mm]
u(\ver_i)= h_i, & 1\leq i\leq N,
\end{array}\right.
\end{eqnarray}
where $\lambda >0$ is a constant, $a_i: \bar E_i \to [0,+\infty)$ is the (possibly degenerate) diffusion,
$H_i: \bar E_i \times \R \to \R$ is the Hamiltonian,
and the Kirchhoff nonlinearity $F:\R^{N+1} \to \R$ is a continuous function satisfying suitable monotonicity
assumptions. As we mentioned above, such kind of problems are inherently discontinuous since no compatibility or continuity
assumptions are made for the $a_i$'s and $H_i$'s at the junction point. Moreover, though the junction condition $F$ is a device to get such a continuity in the ``transition" from edge to edge, it may be of quite different behavior if we compare with the equations near $O$ in each edge.

More precise notations and assumptions are given in Section~\ref{secdefjunction} but let us
immediately give
typical examples of Kirchhoff type nonlinearities. The classical Kirchhoff condition is
\begin{eqnarray}\label{class-kirch}
\sum_{1\leq i\leq N} -u_{x_i} (O) = B,
\end{eqnarray}
where $B\in \R$, but we can also have more general conditions like
$$\alpha_0 u + \sum_{1\leq i\leq N} -\alpha_i u_{x_i} (O) = B,$$
where $\alpha_0\geq 0$ and $\alpha_i >0$ for $1\leq i \leq N$, and even nonlinear ones
$$ \alpha_0 u + \sum_{1\leq i\leq N} \alpha_i (-u_{x_i} (O))^+ +\beta_i (-u_{x_i} (O))^-  = B,$$
with the same conditions as above and with $\beta_i >0$ for $1\leq i \leq N$. Here, for $s\in \R$, we use the notations
$s^+=\max(s,0)$ and $s^-=\min(s,0)$.

In general, problems like~\eqref{pde-junct} do not admit smooth solutions (in
particular if $a_i \equiv 0$) and we are going to use throughout this
paper some suitable notion of viscosity solutions.
In Section~\ref{sec:def-visco}, we give a precise definition for this notion of solution, and in particular at the junction point $O$:
this notion of solutions is the one which is used by Lions and Souganidis~\cite{ls16,ls17} and which is called ``junction viscosity solutions''
in the book of Barles and Chasseigne~\cite{bc24}. 

The study of optimal control problems on networks was the first motivation to look at such nonlinear equations and it naturally began
by the case of first-order Hamilton-Jacobi Equations. In this context, anyway,
the Kirchhoff condition~\eqref{class-kirch}
does not arise as the natural junction condition.
The analysis of deterministic optimal control problems posed on a simple junction and its relation with time-dependent Hamilton-Jacobi Equations can be found in the seminal works Achdou et al.~\cite{acct13} (see also~\cite{aot15}) and Imbert et al.~\cite{imz13}.
Then, Imbert and Monneau
in~\cite{im17} introduced the notion of \textsl{flux-limited solutions}, a particular notion at the junction point whose properties
makes it compatible with the optimal control perspective, and allows to weaken some of the standing assumptions
on the problem. In parallel, Lions and Souganidis~\cite{ls16, ls17}, see also Morfe~\cite{morfe20}, adressed the problem with a PDE point of view
and provide general comparison results for Hamilton-Jacobi Equations with Kirchhoff-type condition on the junction (as in~\eqref{pde-junct}). Surprinsingly enough, in the case of quasiconvex Hamiltonians the two notions of junction condition (flux-limited solutions and
junction viscosity solutions in the terminology of~\cite{bc24}) end up to be equivalent, after a translation of the \textsl{limiter parameters} involved in each formulation. The interested reader is referred to~\cite{im17, ls16, ls17} for the initial efforts to relate both theories, to~\cite{bc24} for a full description of the interplay among these notions of solutions and to~\cite{sc13, siconolfi22} for equations on general networks. 

Concerning nonlinear second-order equations, the structural assumptions at the junction point $O$ addressed in previous works are basically two: either the equation is uniformly elliptic, or the ellipticity vanishes at $O$. In the first case, the regularizing effect of the operator implies the solutions are $C^{2, \alpha}$ and satisfy the Kirchhoff condition in the classical sense, see~\cite{vonbelow88, cms13, cm16, ls16, ohavi21, adlt19, adlt20, blt24}. 
Now, when the ellipticity degenerates at $O$, the
junction condition does not involve second-order terms and both notion
of viscosity solutions for first-order equations can be implemented
for this problem: see Imbert and Nguyen~\cite{in17} for ``flux-limited
solutions'' and Lions and Souganidis~\cite{ls16} for ``junction viscosity
solutions''.


This work provides the first result of well-posedness
for general degenerate second-order nonlinear equations
on networks. We are able to deal with equations whose nature at the
junctions ranges from first-order Hamilton-Jacobi Equations to uniformly elliptic second-order equations
without assuming that either the equations are uniformly elliptic
in {\em each} edge, or the equations
are totally degenerate at the junction point, i.e., {\em all} the diffusions vanish there.

The difficulties of treating the possible degenerate case can be described
as follows.
On the one hand, 
the regularity of the solution to~\eqref{pde-junct} may fail at the junction point $O$
as soon as the diffusion $a_i$ degenerates along a single edge $E_i$
making useless the theory of elliptic equations used in the previous works.
On the other hand, if there is only one edge $E_i$ where the diffusion
ai does not vanish at the junction point $O$, then the viscosity
solutions' approach developed for first-order equations do not work; in
particular, there is no straightforward notion of flux limited solutions
and all the methods for junction viscosity solutions have to be revisited.


Our starting point was the work~\cite{blt24} in which we
realize that, if the diffusion is weak enough --in~\cite{blt24}, we
consider ``weak'' nonlocal diffusion--, then it is possible to
establish the well-posedness of viscosity solutions and to define
a notion of flux limited solution.
This was possible since the problem is, roughly speaking,
close to a first-order problem at the junction point.
We then notice that, for general problems involving second-order
equations, we can come back to a situation which is technically close to
the first-order one for proving comparison results; this is the
purpose of Propositions~\ref{reform-junct} and~\ref{junct-subdiff}
where it is proved that, either the
Kirchhoff condition holds, or only the edges where the diffusion
vanishes at $O$ play a role.
This and other important properties for sub and supersolutions at the junction point are collected in the key Lemma~\ref{inegGi}.
This opens the way to prove a strong comparison result,
i.e., comparison between any discontinuous viscosity subsolution $u$ and supersolution $v$,
see Theorem~\ref{teo:comp}.
The main novelties in the proof of the comparison result are:
on the one hand, we take profit of the first-order behavior of the system at the junction
point, which allows to adapt some ideas Lions and Souganidis~\cite{ls16, ls17} introduced
for first-order Hamilton-Jacobi equations. On the other hand,
suprisingly enough, when $u-v$ achieves a maximum at $O$, then
both the subsolution $u$ and the supersolution $v$ are differentiable at $O$ in nondegenerate
edges. Then, Hopf Lemma permits to erase the contribution of such edges at the
junction point. This allows us to conclude through the Kirchhoff condition, see the proof of Theorem~\ref{teo:comp} for details.

Once comparison is established, 
well-posedness and further properties easily follow by
more classical arguments, see Sections~\ref{sec:exis-junct},~\ref{sec:vanishing-visco} and~\ref{sec:reg}.
We stress once again that even in the case of a single interior vertex
on which we focus here, most of the main difficulties of the general network case are already present.
See Section~\ref{sec:network} for the general network case.

We believe that this work will be the starting point for further
developments. We mention open problems and perspectives in Section~\ref{easytricky}.
In particular, the parabolic version of~\eqref{pde-junct} is a natural extension that requires new ideas since in this case the junction condition must be analyzed on an ``interface" such as line segments $\{ O \} \times (0,T)$ with $T > 0$. 
More generally,
the viscosity solution's approach has shown to be successful to extend the analysis to
general networks with components which may be of higher dimensions, we refer to \cite{aot15}
for first-order equations and to \cite{in17} for second-order ones but that degenerate on the junction set.
We present some extensions of our results in this direction in Section~\ref{easytricky}.
Stratified media were introduced in Bressan and Hong~\cite{bh07}, and we refer to the book of Barles and Chasseigne~\cite{bc24}
for recent developments in this setting.

The paper is organized as follows. In Section~\ref{secdefjunction} we provide the details of the definition of the junction and the PDEs we consider, as 
well as the standing assumptions. In Section~\ref{sec:def-visco} we introduce the notion of junction viscosity solution and provide
some properties of the solutions at the vertices.
Section~\ref{sec-comp} is devoted to the strong comparison result.
Finally, in Section~\ref{vre}, we give the existence and regularity results, we consider the convergence of the vanishing viscosity method and we list some possible extensions. In Appendix~A we present the proofs of some auxiliary results that we use in the body of the paper.
\bigskip

\paragraph{\bf Acknowledgement.}
GB and OL are partially supported by the ANR (Agence Nationale de la Recherche) through the COSS project ANR-22-CE40-0010 and OL by the Centre Henri Lebesgue ANR-11-LABX-0020-01. ET was supported by CNPq Grant 306022/2023-0, and FAPERJ APQ1 210.573/2024. ET and OL were supported by CNPq Grant 408169/2023-0.

\section{Equation on a junction with Kirchhoff condition and Dirichlet boundary condition.}
\label{secdefjunction}

\subsection{Network and Junction}\label{sec:junct1}

As we have seen in the introduction, we have a natural parametrization of each $E_i$ by arc length, namely $x_i=|x|$ on $E_i$. If $\ell_i=|\ver_i|$ and 
$J_i = (0, \ell_i)$, we denote by $\gamma_i : \bar J_i=[0, \ell_i] \to \bar E_i$ the canonical bijection $\gamma_i(x_i) = x$
if $x \in \bar E_i$.
In particular, for each $x \in \Gamma \setminus \{ O \}$, there exists a unique $i$ such that $x \in E_i$, and in this case we write $x_i = 
\gamma_i^{-1}(x) \in J_i$. Throughout the article we will often make the abuse of notation by identifying $x\in E_i$ and $x_i\in J_i$.

Next we introduce the geodesic distance $\rho$ on $\Gamma$. Given $x,y\in \Gamma$, we define
\begin{equation}\label{def-geod}
\rho(x,y) = \left \{ \begin{array}{ll} |x - y| \quad & \text{if $x,y\in \bar E_i$}, \\
|x| + |y| \quad &  \text{if $x\in \bar E_i$, $y\in \bar E_j$ with $i\neq j$}. \end{array} \right . 
\end{equation}
Taking into account the way we parametrize the edges, we have $\rho(x,y) = |x_i - y_j|$ if $i = j$, and $\rho(x,y) = x_i + y_j$ if $i \neq j$. Thus, by abuse of notation, by writing $\rho(x_i, y_j)$ we mean  $\rho(\gamma_i(x_i), \gamma_j(y_j))$ for $x_i \in \bar J_i$ and $y_j \in \bar J_j$.
Notice that the geodesic distance $\rho$ and the distance induced on $\Gamma$ by the Eulidean norm are equivalent.

\subsection{Function spaces.}\label{sec:fct-space}

For a function $u : \Gamma \to \R$, we define $u_i = u \circ \gamma_i: \bar J_i \to \R$,
from which we have $u(x) = u_i(x_i)$ for $x\in \bar E_i$.
The function $u_i$ depends on the parametrization we chose but not its regularity.

We denote by $USC(\Gamma)$ (respectively  $LSC(\Gamma)$) the subset
of functions $u: \Gamma \to \R$ which are upper-semicontinuous (respectively
lower-semicontinuous) on $\Gamma$, that is, for all $i$, $u_i\in USC([0,\ell_i])$
(respectively  $u_i\in LSC([0,\ell_i])$).
The subset  $C(\Gamma)= USC(\Gamma)\cap LSC(\Gamma)$ is the set
of continuous functions on $\Gamma$.
It coincides with the usual notion of continuity
on $\Gamma$ induced by the geodesic distance.

For $u: \Gamma \to \R$, 
we say that $u$ is differentiable at $x \in E_i$ if
$u_i$ is differentiable at $x_i$, and in that case we denote
$$
u_{x}(x) = u_{x_i}(x)=u_i'(x_i).
$$
For $m\in\mathbb{N}$, the space of $m$-times continuously differentiable functions on $\Gamma$ is defined by
\[
C^{m}\left(\Gamma\right):=\left\{ u\in C\left(\Gamma\right):u_i\in C^{m}([0,\ell_i])\text{ for all } i\right\}.
\]
Notice that $u\in C^{m}\left(\Gamma\right)$ is assumed to be continuous on $\Gamma$, all the $u_i$ are
$C^{m}$-continuously differentiable inside the edges and all their derivatives of order less than $m$ can be extended by
continuity to $[0,\ell_i]$. More precisely, when $u\in C^1\left(\Gamma\right)$, we define 
\begin{eqnarray*}\label{prolong-derivative-O}
&&  u_{x_i}(O) := \lim_{x \to O, x \in E_i} u_x(x)
  = \lim_{x \to O, x \in E_i} \frac{u(x) - u(O)}{\rho(x, O)},\\
\label{prolong-derivative-bd}  
&&  u_{x_i}(\ver_i  ) := \lim_{x \to \ver_i, x \in E_i} u_x(x)
  = \lim_{x \to \ver_i , x \in E_i} - \frac{u(x) - u(\ver_i)}{\rho(x, \ver_i)}.
\end{eqnarray*}
The above derivatives depend on the parametrization through the
orientation we chose for the edges (notice the minus sign in the definition of the second one).

We do not try to define here the derivatives arising in the Kirchhoff condition more intrinsically but we point out that we may also use either
the notion of \textsl{inward or outward derivative} with respect to $E_i$ at $O$; these derivatives have the advantage to be independent of
the chosen orientation of the edge in a general network.

\subsection{Hamiltonians.}\label{sec:hamilt}

We assume the existence of a collection of continuous Hamiltonians $\{ H_i \}_{1 \leq i \leq N}$
such that, for all $i$ and for all  $x, y \in \bar E_i$ and $p,q \in \R$, $H_i$ satisfies 
\begin{eqnarray}\label{Ham}
&& \begin{array}{ll}
(i) & |H_i(x,p) - H_i(y, p)|\leq C_H(1+|p|)|x-y|,\\[2mm]
(ii) & |H_i(x,p) - H_i(x, q)|\leq C_H |p - q|. \\[2mm]
\end{array}
\end{eqnarray}

We will also require some coercivity assumption in some cases, namely  
\begin{align}\label{Hcoercive}
H_i(x, p) \geq C_H^{-1}|p| -C_H, \quad \mbox{for $x \in \bar E_i, p \in \R$}.
\end{align}

For simplicity, we have chosen to deal with Hamiltonians satisfying the most classical assumptions which are used in standard Optimal
Control problems, even if we do not impose any convexity assumption on the $H_i$. In particular,
Assumptions~\eqref{Ham}\,(i)-(ii) are related to the boundedness and Lipschitz continuity of the dynamic and the running cost, while
the coercivity assumption~\eqref{Ham} holds true in controllable situations (see, for instance, Bardi and Capuzzo Dolcetta~\cite{bcd97} or~\cite{barles13}). Nevertheless, several results presented here can be readily applied to Hamiltonians with superlinear growth in the gradient,
and/or less regularity on the state variable.

\subsection{Degenerate diffusion.}\label{sec:diff}
We consider a family  $\{ a_{i} \}_{1 \leq i \leq N}$ with
$a_i: \bar E_i\to [0,\infty)$ satisfying
\begin{eqnarray}\label{hyp-diff}
&& \begin{array}{c}
a_i=\sigma_i^2,
\quad
|\sigma_i(x)-\sigma_i(y)|\leq C_a|x-y|, \ x,y\in \bar E_i.
\end{array}
\end{eqnarray}

The above condition on the diffusion coefficient $a$ is a standard assumption on the analysis of degenerate elliptic PDEs in non-divergence form, see~\cite[Theorem II.2]{il90}.

\subsection{Nonlinear Kirchhoff condition.}\label{sec:kirch}

At the junction $O$, we will assume a nonlinear Kirchhoff condition
where
$F: \R\times \R^N\to \R$ is a continuous function satisfying
\begin{eqnarray}\label{hyp-kirch-gene}
&& \begin{array}{cl}
  (i) & \text{for all $r,s\in \R$, ${\bf p}=(p_1,\cdots , p_N), {\bf q}=(q_1,\cdots , q_N)\in \R^N$,}\\
      & \text{if $r\geq s$ and $p_i\leq q_i$ for all $i$, then $F(r,{\bf p})\geq F(s,{\bf q})$,}\\
      & \text{and, if there exists $j$ such that $p_j< q_j$, then $F(r,{\bf p})> F(s,{\bf q})$,}\\[2mm]
 (ii) & \text{for all $i$, $r\in \R$ and $\{p_j\}_{j\not= i}$, $\displaystyle \lim_{p_i\to -\infty} F(r,p_1, \cdots , p_i, \cdots, p_N)=+\infty$.}
\end{array}    
\end{eqnarray}

Typical examples of Kirchhoff conditions are given in the Introduction, see~\eqref{class-kirch} and subsequent examples.

\subsection{The steady assumptions.}

We introduce the set
\begin{eqnarray}
\label{ens-deg}
\mathcal{D}:=\{1 \leq i \leq N:\ a_i(O)=0\}.
\end{eqnarray}

The following {\bf steady assumption} for problem~\eqref{pde-junct}
is in force in all the paper:
\begin{eqnarray}\label{steady}
&& \hspace*{0.2cm}\begin{array}{l}
\text{(i) $\lambda >0$ and $h_i\in \R$, $1\leq i\leq N$, are given},\\[2mm]  
\text{(ii) For all $1\leq i\leq N$, $H_i$ satisfies~\eqref{Ham} and $a_i$ satisfies~\eqref{hyp-diff}},\\[2mm]
\text{(iii) If $i\in\mathcal{D}$, then $H_i$ satisfies~\eqref{Hcoercive}.}\\[2mm]
\text{(iv) The nonlinear Kirchhoff condition satisfies \eqref{hyp-kirch-gene}.}
\end{array}
\end{eqnarray}
Note that we can deal with PDEs which are degenerate in some edges and elliptic in
the others.
Assumption~\eqref{steady}\,(iii) means that, when the PDE is elliptic in an $\bar E_i$-neighborhood of the junction point,
then we can weaken the assumption on the corresponding Hamiltonian $H_i$, that is, coercivity is not
needed anymore.

\section{Definition and properties of viscosity solutions at the vertices}
\label{sec:def-visco}

Using the definitions introduced in the previous section, we are
now in position to
introduce the notion of solution for Problem~\eqref{pde-junct} at the junction.
The case of general networks can be adapted accordingly.

We point out that, in this section, we just assume that $H_i, a_i$
are continuous and $a_i\geq 0$, except in Proposition~\ref{PropDir}.

To simplify the notations, from now on, we denote
\begin{eqnarray}\label{def-G0}
  && G_{i}(u,p,X,x) := 
 \lambda u(x) - a_i(x)X + H_i(x, p),
\end{eqnarray}
$ \mathbf{u_x (O)}:=(u_{x_1}(O), \cdots ,  u_{x_N}(O)) \in \R^N$
and $\1 :=(1,\cdots,1) \in \R^N$.


\begin{defi}\label{defi-visco}
We say that $u \in USC(\Gamma)$  is a viscosity subsolution to problem~\eqref{pde-junct}
if, for each $x \in \Gamma$ and $\varphi \in C^2(\Gamma)$ such that $x$ is a local maximum point of $u - \varphi$, then we have
\begin{eqnarray*}
G_{i}\big(u(x), \varphi_{x_i}(x), \varphi_{x_i x_i}(x), x\big)\leq 0  &\mbox{if} & x \in E_i, 
\\
\min \Big{\{} \min_{1\leq i\leq N} G_{i}\big(u(O), \varphi_{x_i}(O), \varphi_{x_i x_i}(O) ,O\big) , F(u(O),\mathbf{\varphib_{x}(O)}) \Big{\}} \leq 0
& \mbox{if} & x=O,
\\
  \min \Big{\{} G_{i}\big(u(x), \varphi_{x_i}(x), \varphi_{x_i x_i}(x) , x\big)  , \ u(x)-h_i \Big{\}} \leq 0
& \mbox{if} & x=\ver_i.  
\end{eqnarray*}

We say that $v \in LSC(\Gamma)$  is a viscosity supersolution to problem~\eqref{pde-junct}
if, for each $x \in \Gamma$ and $\varphi \in C^2(\Gamma)$ such that $x$ is a local minimum point of
$v - \varphi$, then we have
\begin{eqnarray*}
G_{i}\big(v(x), \varphi_{x_i}(x), \varphi_{x_i x_i}(x),x\big)\geq 0  &\mbox{if} & x \in E_i, 
\\
\max \Big{\{} \max_{1\leq i\leq N} G_{i}\big(v(x), \varphi_{x_i}(O), \varphi_{x_i x_i}(x),O\big) ,
F(v(O),\mathbf{\varphib_{x}(O)} \Big{\}} \geq 0
& \mbox{if} & x=O,
\\
  \max \Big{\{} G_{i}\big(v(x), \varphi_{x_i}(x), \varphi_{x_i x_i}(x),x\big)  , \ v(x)-h_i \Big{\}} \geq 0
& \mbox{if} & x=\ver_i.  
\end{eqnarray*}

A viscosity solution  $u:\Gamma\to \R$ is a locally bounded function such that
its upper-semicontinuous envelope $u^*$ is a viscosity subsolution and
its lower-semicontinuous envelope $u_*$ is a viscosity supersolution.
\end{defi}

We continue with some reformulations of the definition of viscosity solutions at the junction. The following
results shows that, maybe surprisingly, the nondegenerate equations do not appear,
which leads to a first-order junction condition.

\begin{prop}\label{reform-junct}
Let $u \in USC(\Gamma)$  be a viscosity subsolution to problem~\eqref{pde-junct}
and let  $\varphi \in C^2(\Gamma)$ such that $O$ is a local maximum point of $u - \varphi$. Then
\begin{eqnarray}
\label{eq683}
&& \min \Big{\{} \min_{i \in \mathcal{D}}  \left\{\lambda u(O) + H_i(O, \varphi_{x_i}(O))\right\} , F(u(O),\mathbf{\varphib_{x}(O)} \Big{\}} \leq 0.
\end{eqnarray}

Let $v \in LSC(\Gamma)$  be a viscosity supersolution to problem~\eqref{pde-junct}
and let  $\varphi \in C^2(\Gamma)$ such that $O$ is a local minimum point of $v - \varphi$. Then
\begin{eqnarray}
\label{eq684}
&&\max \Big{\{} \max_{i \in \mathcal{D}}  \left\{\lambda v(O) + H_i(O, \varphi_{x_i}(O))\right\} , F(v(O),\mathbf{\varphib_{x}(O)} \Big{\}} \geq 0.
\end{eqnarray}
\end{prop}

\begin{rema} \ \\
(i) As we already mentioned it above, Proposition~\ref{reform-junct} shows that, in the junction condition, only the equations which are 
degenerate at $O$ play a role. In particular, if all the equations are uniformly elliptic at $O$, then
the Kirchhoff condition holds in a strong sense.\\
(ii) We point out that Proposition~\ref{reform-junct} allows to reduce the definition to first-order conditions at $O$
for both the sub and supersolution.
Proposition~\ref{junct-subdiff} below will even make this more precise by proving that the solutions behave
really as solutions to a first-order equation at the junction $O$.
\end{rema}
  
\begin{proof}
We provide the proof for subsolutions, the one for supersolutions being similar.

Let $\epsilon, L >0$ and define $\psi\in C^2(\Gamma)$ by $\psi(x):= \epsilon \rho(x,O)-L \rho(x,O)^2$,
where $\rho$ is defined in~\eqref{def-geod}.
Since $\psi(O)=0$ and $\psi(x)>0$ if $x$ is close enough to $O$, the function $u-\varphi -\psi$
still achieves a maximum at $O$. It follows from Definition~\ref{defi-visco} that
\begin{eqnarray}\label{def-visc-psi}
 && \min \Big{\{} \min_{1\leq i\leq N} G_{i}\big(u(O), \varphi_{x_i}(O)+\psi_{x_i}(O), \varphi_{x_i x_i}(O)+\psi_{x_i x_i}(O) ,O\big),\\
 && \hspace*{5.5cm} F( u(O),\mathbf{(\varphib+\psib)_{x}(O)} ) \Big{\}} \leq 0.\nonumber
\end{eqnarray}

Using that $\psi_{x_i}(O)= \epsilon$ and $\psi_{x_i x_i}(O)=-2L$, we have,
on one hand,
\begin{eqnarray*}
&&   F(u(O),\mathbf{(\varphib+\psib)_{x}(O)}) = F(u(O),\mathbf{\varphib_{x}(O)}+\e \1).
\end{eqnarray*}

On the other hand,
\begin{eqnarray}\label{estimG_i-psi}
  &&  G_{i}\big(u(O), \varphi_{x_i}(O)+\psi_{x_i}(O), \varphi_{x_i x_i}(O)+\psi_{x_i x_i}(O) ,O\big)\\
 & =& \lambda u(O)-a_i(O)\varphi_{x_i x_i}(O) + 2a_i(O)L + H_i(O, \varphi_{x_i}(O)+\epsilon) \nonumber\\
 & \geq &  G_{i}\big(u(O), \varphi_{x_i}(O), \varphi_{x_i x_i}(O) ,O\big) +  2a_i(O)L +o_\epsilon(1),\nonumber
\end{eqnarray}
thanks to the continuity of $H_i$.

Taking $L$ large enough, for all $1\leq i\leq N$ such that $a_i(O)>0$ and all $\epsilon \in (0,1)$, we obtain that
$$
G_{i}\big(u(O), \varphi_{x_i}(O)+\psi_{x_i}(O), \varphi_{x_i x_i}(O)+\psi_{x_i x_i}(O) ,O\big)>0.
$$

And, if $a_i(O)=0$, then the right-hand side of~\eqref{estimG_i-psi} reads
$$
\lambda u(O)+ H_i(O, \varphi_{x_i}(O)) + o_\epsilon(1).
$$

Hence, \eqref{def-visc-psi} reduces to
$$ \min  \Big{\{} \min_{i\in \mathcal{D}}\left\{ \lambda u(O)
+ H_i(O, \varphi_{x_i}(O)) + o_\epsilon(1)\right\}, F(u(O),\mathbf{\varphib_{x}(O)}+\e \1)  \Big{\}} \leq 0,$$
and the result follows by letting $\epsilon$ tend to $0$.
This ends the proof.
\end{proof}

As usual, it is possible to state an equivalent definition
to Definition~\ref{defi-visco} and~\eqref{eq683}-\eqref{eq684}
in terms of sub/superjets,
the definition of which is recalled now (see~\cite{cil92} for details).

Given an interval $I \subset \R$ and $u\in USC(I)$, we recall that $(p,X) \in \R\times \R$ is in the superjet
$J^{2,+}_I u(x_0)$ of $u$ at $x_0 \in I$ if
$$
u(x) \leq u(x_0) + p(x - x_0) + \frac{1}{2} X (x-x_0)^2 + o(|x - x_0|^2),
$$
for all $x$ in an $I$-neighborhood of  $x_0$. 
In the case of junctions, we say that $({\bf p},{\bf X})  = ((p_1, \cdots , p_N), (X_1, \cdots , X_N)) \in \R^N\times \R^N$
is in the  superjet $J^{2,+}_\Gamma u(O)$
of $u\in USC(\Gamma)$ at $O$ if, for each $1 \leq i \leq N$, then
$(p_i,X_i)\in J^{2,+}_{\bar J_i} u_i(0)$, i.e.,
$$
u_i(x_i) \leq u(O) + p_i x_i + \frac{1}{2} X_i x_i^2 + o(x_i^2),$$
for all $x_i$ in an $\bar J_i$-neighborhood of $0$.
Similar definitions can be stated for the subjet $J^{2,-}_\Gamma v(O)$ of $v\in LSC(\Gamma)$ by reversing the inequalities.

Let $u \in USC(\Gamma)$ and $v \in LSC(\Gamma)$ and denote
$$
\mathbf{\varphib_{xx}(O)}:=(\varphi_{x_1 x_1}(O),\cdots ,  \varphi_{x_N x_N}(O))
$$ 
for $\varphi\in C^2(\Gamma)$.
Using the characterization 
\begin{eqnarray}
  \label{super-equiv}  \hspace*{0.9cm}J^{2,+}_\Gamma u(O) & = & \big\{(\mathbf{\varphib_{x}(O)},\mathbf{\varphib_{xx}(O)})
 :\\
&& \nonumber \hspace*{0.5cm}\text{$\varphi\in C^2(\Gamma)$ and $u-\varphi$ has a local maximum at $O$} \big\},\\
\label{sub-equiv}  J^{2,-}_\Gamma v(O) & = & \big\{(\mathbf{\varphib_{x}(O)}, \mathbf{\varphib_{xx}(O)}):\\
&& \nonumber \hspace*{0.5cm}\text{$\varphi\in C^2(\Gamma)$ and $v-\varphi$ has a local minimum at $O$} \big\},
\end{eqnarray}
it is possible to state an equivalent definition of viscosity sub/supersolutions
in terms of sub/superjets. We leave the statement to the reader.

From now on, for the sake of notations, when there is no ambiguity, we will omit
  the dependence of the sub/superjets with respect to $\Gamma$,
  writing simply  $J^{2,+} u(O)$, $ J^{2,-} v(O)$.


Proposition~\ref{reform-junct} shows that, at $O$, the viscosity inequalities can be expressed by using only the
first derivatives of the test-functions. A natural (but not straightforward) consequence is that the definition of viscosity
sub- and supersolutions at $O$ can be expressed only in terms of $C^1$ test-functions, or equivalently
(first-order) sub/superdifferentials.

We first define precisely the sub/superdifferentials at $O$ in a similar way as the sub/superjets.
If $u \in USC(\Gamma)$, then we say that ${\bf p} = (p_1, \cdots , p_N)$ is in the  superdifferential $D^+ u(O)$
of $u$ at $O$ if, for each $1 \leq i \leq N$, we have
$$
u_i(x_i) \leq u(O) + p_i x_i + o(x_i) \quad \mbox{for all $x_i$ in an $\bar J_i$-neighborhood of  $0$.} 
$$

The subdifferential  $D^- v(O)$ of $v\in LSC(\Gamma)$ at $O$ is defined by reversing the inequalities.

Similarly to~\eqref{super-equiv}-\eqref{sub-equiv}, we have
\begin{eqnarray*}
&& D^{+} u(O)=\big\{\mathbf{\varphib_{x}(O)}:\ \text{$\varphi\in C^1(\Gamma)$ and $u-\varphi$ has a local maximum at $O$} \big\},\\
&& D^{-} v(O)=\big\{\mathbf{\varphib_{x}(O)}:\ \text{$\varphi\in C^1(\Gamma)$ and $v-\varphi$ has a local minimum at $O$} \big\}.
\end{eqnarray*}
Note that the above property may not be true with a $C^2$-test-function (it is possible for instance
that $ D^{+} u(O)\not=\emptyset$ while $J^{2,+} u(O)=\emptyset$).
However, it is possible to characterize the junction conditions~\eqref{eq683}-\eqref{eq684}
using only $C^1$ test-functions or, equivalently, (first-order) sub/superdifferentials. This is the aim of the following result.

\begin{prop}\label{junct-subdiff}
Let $u \in USC(\Gamma)$  be a viscosity subsolution to problem~\eqref{pde-junct}
and let ${\bf p} = (p_1, \cdots , p_N)\in D^+ u(O)$. Then
\begin{eqnarray}
\label{eq683bis}
&& \min \Big{\{} \min_{i \in \mathcal{D}}  \left\{\lambda u(O) + H_i(O, p_i)\right\} , F(u(O),\mathbf{p}) \Big{\}} \leq 0.
\end{eqnarray}

Let $v \in LSC(\Gamma)$  be a viscosity supersolution to problem~\eqref{pde-junct}
and let  ${\bf q} = (q_1, \cdots , q_N)\in D^- v(O)$. Then
\begin{eqnarray}
\label{eq684bis}
&&\max \Big{\{} \max_{i \in \mathcal{D}}  \left\{\lambda u(O) + H_i(O, q_i)\right\} , F(v(O),\mathbf{q})  \Big{\}} \geq 0.
\end{eqnarray}
\end{prop}

\begin{proof} We provide the proof only in the subsolution case, the supersolution one being analogous with obvious adaptations.

If ${\bf p} = (p_1, \cdots , p_N) \in D^+ u(O)$, then, for each $1 \leq i \leq N$, we have
$$
u_i(x_i) \leq u(O) + p_i x_i + o(x_i) \quad \mbox{for all $x_i$ in an $\bar J_i$-neighborhood of  $0$.} 
$$
Pick any $X \in \R$, $0<\e \ll 1$ and define $\varphi(x)$ by
$$
\varphi_i (x_i)= u(O) + (p_i +\e) x_i +\frac12 Xx_ix_i,  \quad \text{for any $x_i\in\bar J_i$, $1 \leq i \leq N$.}
$$
We have, in an $\bar J_i$-neighborhood of  $0$,
\begin{align*}
u_i(x_i)-\varphi_i (x_i)= & u_i(x_i)-u(O) - (p_i +\e) x_i - \frac12 Xx_ix_i \\
\leq & -\e x_i - \frac12 Xx_ix_i + o(x_i)
\leq 0 = u_i(0)-\varphi_i (0),
\end{align*}
if $x_i$ is small enough.
Hence $O$ is a maximum point of $u-\varphi$ and, since $\varphi_{x_i}(O)=p_i+\e$, we obtain
$$\min \Big{\{} \min_{i \in \mathcal{D}}  \left\{\lambda u(O) + H_i(O, p_i+\e)\right\} , F(u(O),\mathbf{p}+\e\1) \Big{\}} \leq 0.
$$
And the conclusion follows by letting $\e$ tend to $0$.
\end{proof}

We conclude this section with some remarks on the Dirichlet problem at each $\ver_i\in \VVb$. Since the equations on $E_i$ can be degenerate,
the Dirichlet boundary conditions may not be satisfied in a classical sense. Of course, we use the notion of boundary condition in the viscosity sense;
we are not going to rewrite here the definition of such boundary conditions and we just refer the reader to the User's guide of Crandall, Ishii and 
Lions~\cite[Section 7]{cil92} for a complete description of them. Instead we provide various properties under suitable assumptions.
\begin{prop}\label{PropDir}Assume~\eqref{steady}. Let $u \in USC(\Gamma)$ be a viscosity subsolution
and $v \in LSC(\Gamma)$  be a viscosity supersolution
to Problem~\eqref{pde-junct}. Then the followings hold
\begin{itemize}
\item[(i)] if $a_i(\ver_i)>0$, then $u(\ver_i) \leq h_i \leq v(\ver_i)$,\\
\item[(ii)] If $a_i(\ver_i)=0$ and if \eqref{Hcoercive} holds in a neighborhood of $\ver_i$, then $u(\ver_i) \leq h_i$. Moreover
if $\tilde u:\Gamma \to \R$ is defined by
$$
\tilde u(x):= 
\begin{cases}
u(x) & \hbox{if } x\notin \VVb,\\
{\displaystyle \limsup_{x \to \ver_i, x\in E_i} u(x)} & \hbox{if }x=\ver_i, \; \hbox{for any $1\leq i \leq N$,},
\end{cases}
$$
then $\tilde u$ is still a subsolution of Problem~\eqref{pde-junct}.
\end{itemize}
\end{prop}

\begin{proof} To prove (i) and the first part of (ii), we consider the function $\psi (x)= L(\rho(x,\ver_i)-K \rho(x,\ver_i)^2)$
of Lemma~\ref{fct-test-quad} (with $y$ replaced by $\ver_i$ there)
for some constants $L,K$ large enough,
and we look for maximum points of $u(x)-\psi (x)$ for $x$ such that $0\leq \rho(x,\ver_i) \leq (4K)^{-1}$.
We first choose $K$ large enough to have $(4K)^{-1}< \rho(\ver_i,O)$; in that way, the maximum points of
$u-\psi$ are necessarily in $\bar E_i \setminus \{O\}$
and $\psi (x)= L((\ell_i -x_i)- K (\ell_i -x_i)^2)$.

Using Lemma~\ref{fct-test-quad},
if $L$ is large enough, then
we can write down a viscosity subsolution inequality, either if $\bar x\in E_i$, or $\bar x=\ver_i$. In both cases, the term
coming from the equation reads
\begin{eqnarray*}
&& \lambda u(\bar x)-a_i(\bar x) \psi_{x_i x_i} (\bar x) + H_i(\bar x, \psi_{x_i}(\bar x))\\
&=& \lambda u(\bar x)+ 2a_i(\bar x) LK + H_i(\bar x, -L + 2LK(\ell_i -\bar x_i)).
\end{eqnarray*}
Hence there are two cases:

\smallskip

\noindent
(a) If $a_i(\ver_i)>0$, since $\bar x_i \to \ell_i$ when $L\to +\infty$ (see~\eqref{psixbarre}), then the above quantity cannot be less than $0$ for $L$ large enough if $K$ is
chosen sufficiently large (with a size depending only $H_i$ and $a_i$). Therefore, we necessarily have $\bar x = \ver_i$ when $L$ is large
enough, and the boundary condition inequality reads $u(\ver_i) \leq h_i$.

\smallskip

\noindent
(b) If $a_i(\ver_i)=0$, then we just use that $2 a_i(\bar x) LK\geq 0$ (here the choice of $K$ is not important)
and $\vert \psi_{x_i}(\bar x) \vert \geq L/2$ (see~\eqref{gradpsi}). Using \eqref{Hcoercive}, it implies that the above 
expression cannot be less than $0$ and again we necessarily have $\bar x = l_i$ and $u_i(\ell_i) \leq h_i$.

The proof of the inequality $h_i \leq v(\ver_i)$ follows from an easy adaptation of the above argument when $a_i(\ver_i)>0$. 

Concerning $\tilde u$, it is clear that it is still a subsolution of Problem~\eqref{pde-junct} since, for any $i$,
$\tilde u(\ver_i) \leq u(\ver_i) \leq h_i$\footnote{The reader may remark that, if $\tilde u(\ver_i) < h_i$, then the subsolution inequality does not impose sufficient constraint on the subsolution and it may take any ``artificial'' value in the interval $[\tilde u(\ver_i), h_i]$.}.
\end{proof}

Motivated by the above result, we use in the sequel the assumption
\begin{eqnarray}\label{hyp-dir}
&& \\
\nonumber && \hbox{For any $i$, either $a_i(\ver_i)>0$, or \eqref{Hcoercive} holds in a neighborhood of $\ver_i$. }
\end{eqnarray}
\section{Comparison for the degenerate Kirchhoff-type problem on a junction}
\label{sec-comp}

\begin{teo}\label{teo:comp}
Assume~\eqref{steady} and \eqref{hyp-dir}. Let $u \in USC(\Gamma)$ be a viscosity subsolution
and $v \in LSC(\Gamma)$  be a viscosity supersolution
to Problem~\eqref{pde-junct}.
If $\tilde u$ is the function defined in Proposition~\ref{PropDir} then $\tilde u\leq v$ on $\Gamma$. In particular,
$u \leq v$ on $\Gamma \setminus \VVb$.
\end{teo}
\begin{rema} The admittedly strange form of Theorem~\ref{teo:comp}, and in particular the comparison of $u$ and $v$ which
only holds on $\Gamma \setminus \VVb$ is connected to the usual difficulty arising in the Dirichlet problem for degenerate equations where
losses of boundary conditions may arise. We do not insist on that feature since it is not the main new point in this result and we refer to 
\cite{bp90, bb95, bdl04, br98} where such problems are studied in different contexts.

\end{rema}

The whole section is devoted to the proof of Theorem~\ref{teo:comp}
adapting some ideas introduced in~\cite{ls17} and put in use in~\cite[Section 15.3]{bc24}. 

To simplify the notations, we write $u$ instead of $\tilde u$ and, as usual, we argue by contradiction, assuming that
$$
M=\sup_{\Gamma} \{ u - v \} > 0.
$$
Since $\Gamma$ is compact, the supremum is attained at some point $x_0 \in \Gamma$. 

If $x_0 \in \Gamma \setminus \VV$, where $\VV$ is the set of vertices, then standard arguments
using the classical doubling of variables apply and provide the result  (see~\cite{cil92}).

If $x_0 = \ver_i$ for a boundary vertex $\ver_i\in \VVb$, we remark that necessarily $a_i(\ver_i)=0$; otherwise we would have
$u(\ver_i)-v(\ver_i)\leq 0$ by Proposition~\ref{PropDir}. Using again Proposition~\ref{PropDir}, we have $u(\ver_i) \leq h_i$ and therefore
$v(\ver_i)<h_i$. Hence the boundary condition in the viscosity sense for $v$ at $\ver_i$ reduces to 
$$ \lambda v(\ver_i)+ H_i(\ver_i, v_x) \geq 0,$$
a boundary condition of state-constraint type. Since we are in a $1$-d situation and since $u(\ver_i)={\displaystyle \limsup_{x \to \ver_i, x\in E_i} u(x)}$,
an easy adaptation of the proof of \cite[Theorem 7.9]{cil92} leads to a contradiction.

Therefore, the specific and difficult case we will focus on is when $x_0 = O$
is the unique maximum point of $u-v$, i.e., 
\begin{equation*}
M = (u - v)(O) > (u-v)(x), \quad x\in\Gamma\setminus \{O\}.
\end{equation*}

Before continuing the proof of the theorem, we first
give some technical results about the viscosity inequalities which are
satisfied at the junction on each branch in terms of sub- and superdifferentials of the functions $u$ and $v$, respectively.

To do so, for each $i$, we define
\begin{eqnarray*}
&& \bar p_i = \limsup_{x_i \to 0^+} \frac{u_i(x_i) - u_i(0)}{x_i}, \qquad \underline p_i = \liminf_{x_i \to 0^+} \frac{u_i(x_i) - u_i(0)}{x_i}, \\
&& \bar q_i = \limsup_{x_i \to 0^+} \frac{v_i(x_i) - v_i(0)}{x_i}, \qquad \underline q_i = \liminf_{x_i \to 0^+} \frac{v_i(x_i) - v_i(0)}{x_i}.
\end{eqnarray*}


\begin{lema}\label{inegGi}
Assume~\eqref{steady}. Let $u \in USC(\Gamma)$ be a viscosity subsolution
and $v \in LSC(\Gamma)$  be a viscosity supersolution
to problem~\eqref{pde-junct}
such that $\max_{\Gamma} (u-v) =(u-v)(O)$.
Then

\begin{itemize}

\item[(i)] {\bf (Subsolutions)} For all $1\leq i\leq N$, $\underline p_i$ and $\bar p_i$ are finite, and $D^+u_i(O)=[\bar p_i ,+\infty)$.

\item[(ii)] {\bf (Supersolutions)}  For all $1\leq i\leq N$, $-\infty < \underline p_i \leq \underline q_i$,
and $D^-v_i(O)=(-\infty, \underline q_i]$ if $\underline q_i <+\infty$, $D^-v_i(O)=\R$ otherwise.
  
\item[(iii)] {\bf (Degenerate branches)} If $i \in \mathcal{D}$, then, for any $p \in [\underline p_i, \bar p_i ]$,
\begin{eqnarray}\label{ineg-sous-deg}
    && \lambda u(O)+ H_i(O, p) \leq 0.
\end{eqnarray}
  In the same way, for the supersolution,  for any $q \in (\underline q_i, \bar q_i )$, we have
\begin{eqnarray}\label{ineg-sur-deg}
    && \lambda v(O)+ H_i(O, q) \geq 0.
\end{eqnarray}
\item[(iv)] {\bf (Nondegenerate branches)} If $i \notin \mathcal{D}$, then $\underline p_i = \bar p_i$ and $\underline q_i = \bar q_i$.

\end{itemize}
\end{lema}

\begin{rema} \label{loss-bd}
Lemma~\ref{inegGi} is a second-order analogue to Lions-Souganidis arguments~\cite{ls16,ls17} (see also~\cite[Lemma 15.1, p.258]{bc24}).
Note that, for a subsolution $u$, both $\underline p_i$ and $\bar p_i$ are finite.
For a supersolution $v$, we know that $-\infty <\underline q_i$ when $\max_{\Gamma} (u-v)= (u-v)(O)$ (i.e., when
we can ``touch'' $v$ from below at $O$)
but $\underline q_i$ and/or $\bar q_i$ may be $+\infty$.
\end{rema}  


\begin{proof}
To prove (i), we first claim that, for $L,K>0$ large enough, then
\begin{eqnarray}\label{claimLK}
  && u(x) \leq u(O) + \psi(x) \quad \text{for $x$ close to $O$,}
\end{eqnarray}
where $\psi$ is the function introduced in Lemma~\ref{fct-test-quad} with $y = O$, that is $\psi(x) = L(\rho(x,O) - K \rho(x, O)^2)$. We consider
\begin{eqnarray*}
  \max_{ \{\rho(\cdot,O)\leq (4K)^{-1}\}} \{ u-\psi \}.
\end{eqnarray*}

From Lemma~\ref{fct-test-quad}, for $L$ large enough depending on $K$ and $\osc (u)$, the maximum is achieved at some $\bar x$ such that $\rho(\cdot,O)< (4K)^{-1}$.
If $\bar x=O$, then \eqref{claimLK} is proved.
Otherwise, up to increase $K$ if needed, $\bar x\not\in  \VVb$, thus
$\bar x\in E_i$ for some $1\leq i\leq N$. We
we can write the viscosity inequality of Definition~\ref{defi-visco} inside $E_i$,
which reads
\begin{eqnarray*}
0 &\geq &  G_{i}\big(u(\bar x), \psi_{x_i}(\bar{x}), \psi_{x_i x_i}(\bar{x}) ,\bar{x}\big)\\
&=& \lambda u(\bar x)+ 2 a_i(\bar{x}_i) KL + H_i(\bar{x}, L(1-2K\bar{x}_i)).
\end{eqnarray*}
Now there are two cases: 

\smallskip

\noindent
-- If $i \in \mathcal{D}$, using that $2 a_i(\bar{x}_i) KL\geq 0$, we have
$$\lambda u(\bar x)+ H_i(\bar{x}, L(1-2K\bar{x}_i))\leq 0,$$
and, by choosing $L$ large enough, we have a contradiction since $L(1-2K\bar{x}_i)\geq L/2$
(see~\eqref{gradpsi}) and $H_i$ is coercive.

\smallskip

\noindent
-- If $i \notin \mathcal{D}$, then, using that $H_i (x,p)$ is Lipschitz continuous (see~\eqref{Ham}), we obtain
$$ \lambda u(\bar x)+ 2 a_i(\bar{x}_i) KL + H_i(\bar{x},0)-C_H L(1-2K\bar{x}_i)) \leq 0.$$
Thanks to the continuity of $a_i$ and the fact that $i\not\in \mathcal{D}$, we first choose
$K$ large enough in order that $\min_{x\in V_K \cap E_i} a_i(x) >0$,
where $V_K= \{\rho(\cdot,O)\leq (4K)^{-1}\}$.
Secondly, enlarging $K$ if needed in such a way
that $$2K\min_{i \notin \mathcal{D}}\min_{x\in V_K\cap E_i} a_i(x) > C_H,$$ and then taking
$L$ large enough, we also have a contradiction and the claim~\eqref{claimLK} is proved.

A straightforward consequence of the claim is that $\bar p_i <+\infty$ for all $i$ and
therefore $D^+ u_i(O)\neq \emptyset$.

Next we prove that $\bar p_i >-\infty$ for all $i$. If not, there exists $i$
such that $\bar p_i=-\infty$ and we introduce the function $\chi$ defined in a neighborhood of $O$ by
$$ \chi(x) :=
\begin{cases}
\psi(x) & \hbox{if  $x\in E_j$ with $j\neq i$,} \\
px & \hbox{if  $x\in E_i$,}
\end{cases}
$$
where $\psi$ is defined at the beginning of the proof, and $p\in \R$.

We claim that, for $L,K>0$ large enough and for any $p\in \R$, then $O$ is a local maximum point of $u-\chi$. Indeed, on one hand,
for $L,K>0$ large enough,
we have $u(x) - \chi(x) \leq u(O)-\chi(O)$  for any $x\in E_j$, $j\neq i$, close enough to $O$
since~\eqref{claimLK} holds by the above;
on the other hand, since $\bar p_i=-\infty$, then $u(x)-px \leq u(0)$ for $x \in E_i$, close enough to $O$.

Moreover, by coercivity of the Hamiltonians, we can choose $L$ large enough in order to have $\lambda u(O)+ H_j(0,L)>0$ for $j \in \mathcal{D}$, $j\neq i$, and, if $i \in
\mathcal{D}$, then we can also assume that $\lambda u(O)+ H_i(0,p)>0$ by choosing $p<0$ with $|p|$ large enough.

With these choices, thanks to Proposition~\ref{reform-junct}, the subsolution inequality obviously reduces to the nonlinear Kirchhoff
condition
\begin{eqnarray*}\label{kirch-contr763}
F(u(O), \mathbf{\chib_{x}(O)}) = F(u(O), (L,\cdots,L,p,L,\cdots,L)) \leq 0,
\end{eqnarray*}
where the $p$ is at the $i^{th}$-position. But using Assumption~\eqref{hyp-kirch-gene}-(ii) and taking $p<0$ with $|p|$ large enough, this inequality cannot hold and we reach a contradiction.

%

The property $D^+ u_i(O)=[\bar p_i, +\infty)$ follows from~\cite[Proposition~2.10, p.~84]{bc24}.
  
Now we prove, at the same time, Property~\eqref{ineg-sous-deg} and $\underline p_i >-\infty$
for $i\in\mathcal{D}$ (we will show later that this property of $\underline p_i$ also holds if $i\not\in\mathcal{D}$).
The argument is exactly the same as in Lions and Souganidis~\cite{ls16,ls17} (see also~\cite[Lemma 15.1, p.258]{bc24}):
if $\underline p_i < \bar p_i$,
it consists in looking at $\max_{x_i \in [0,b_k]} \{ u_i (x_i)-px_i \}$ for
$p\in (\underline p_i , \bar p_i)$ and
a suitable choice of a sequence $(b_k)_k$ converging to $0$,
see Lemma~\ref{lem-max123} (ii) with $X=0$.
The fact that the second derivative of the function $x_i \mapsto px_i$ is $0$
takes care of the $a_i$-term. More precisely, writing the viscosity inequality
for the subsolution at the maximum point $x_k\in (0,b_k)$ given by Lemma~\ref{lem-max123} (ii)
and letting $k\to +\infty$, we get
\begin{eqnarray*}
\lambda u(O)+ H_i(O, p) \leq 0.
\end{eqnarray*}
This proves~\eqref{ineg-sous-deg} for $p\in (\underline p_i , \bar p_i)$
and
the coercivity of $H_i$ implies that all these $p\in (\underline p_i , \bar p_i)$ are uniformly
bounded and therefore necessarily $\underline p_i >-\infty$. And, of course, Inequality~\eqref{ineg-sous-deg} also holds for $p=\underline p_i$
and $p= \bar p_i$ by the continuity of $H_i$.

If $\underline p_i = \bar p_i$, then we consider $u^\e (x_i):= u_i (x_i)+\e x_i\sin(\log(x_i))$.
Defining $\underline p_i^\e,\bar p_i^\e$ in the same way as $\underline p_i,\bar p_i$ above
changing $u_i $ in $u^\e$, we have $\underline p_i^\e < \bar p_i^\e$. 
Then, we can apply Lemma~\ref{lem-max123} (ii) with $w=u^\e$, $p\in (\underline p_i^\e, \bar p_i^\e)$ and $X=0$
to obtain that the function
$$
x_i\mapsto u_i^\e (x_i)- u_i^\e(0)-px_i=u_i(x_i)-u_i(0) - px_i +\e x_i\sin(\log(x_i))
$$
has a sequence of maximum points $0 < x_k \to 0$. 
Writing the viscosity inequality for the subsolution $u$ at $x_k$ inside the edge $E_i$,
 we arrive at
\begin{eqnarray*}
&& \lambda u_i(x_k) + H_i(x_k,p) \leq C \frac{a_i(x_k)}{x_k} + O(\epsilon).
\end{eqnarray*}
Since $a_i(O)=0$, from~\eqref{hyp-diff}, we get $a_i(x_k)=\sigma_i^2(x_k)\leq C_a|x_k|^2$.
Therefore, passing to the limit $x_k\to 0$ and then $\epsilon\to 0$, we conclude~\eqref{ineg-sous-deg}
for $p=\underline p_i = \bar p_i$.

For (ii) in the case of the supersolution, since $O$ is maximum point of $u-v$, we have $-\infty < \underline p_i \leq \underline q_i$. And since $q \in D^-v_i(O)$ for any $q \leq \underline q_i$, we have either $\underline q_i <+\infty$ and $D^-v_i(O)=(-\infty,\underline q_i]$, or $D^-v_i(O)=\R$. In the former case, the Lions-Souganidis argument gives as above the expected inequality~\eqref{ineg-sur-deg}.

We provide the proof of (iv) only in the subsolution case, the supersolution one can be treated analogously.
Note that an immediate by-product is that $\underline p_i >-\infty$ for $i\not\in\mathcal{D}$ since we already know that $\bar p_i >-\infty$.

If $\underline p_i < \bar p_i $, then 
we use the Lions-Souganidis argument with a small modification: instead of looking at
$\max_{x_i \in [0,b_k]} \{ u_i (x_i)-px_i\}$ for $p\in (\underline p_i ,\bar p_i)$
and a suitable choice of a sequence $(b_k)_k$ converging to $0$,
we look at $\max_{x_i \in [0,b_k]} \{u_i (x_i)-px_i+Kx_i^2\}$.
Writing the subsolution inequality at the maximum point $x_k\in (0,b_k)$ given by Lemma~\ref{lem-max123} (ii) with $X=-2K$
and letting $k\to +\infty$, we arrive at
$$ 
\lambda u(O)+2Ka_i(O)+ H_i(O, p)\leq 0.
$$
This cannot hold if $K$ is chosen large enough since $a_i(O)>0$. The proof of Lemma~\ref{inegGi} is complete.
\end{proof}


We are now in position to prove the comparison result.

\begin{proof}[Continuation of the proof of Theorem~\ref{teo:comp}]
As explained above, we argue by contradiction assuming that $O$ is
the unique positive maximum of $u-v$, that is
\begin{equation}\label{contraO}
0 < M :=  (u - v)(O) > (u-v)(x), \quad x\in\Gamma\setminus \{O\}.
\end{equation}

From Lemma~\ref{inegGi} (i),(ii), we know that, for all $1\leq i\leq N$,
$\underline p_i, \bar p_i$ are finite and $\underline p_i \leq \underline q_i$.
And by Lemma~\ref{inegGi} (iv), if $i\notin \mathcal{D}$, then $\underline p_i=\bar p_i \leq \underline q_i=\bar q_i$.

\smallskip

We distinguish two cases.
\smallskip

\noindent
\textsl{Case 1. There exists $ i\in \mathcal{D}$ such that $\underline p_i \leq\underline q_i \leq \bar p_i$.}
\smallskip

\noindent
From Lemma~\ref{inegGi} (iii), we have, for all $p \in [\underline p_i, \bar p_i]$
$$ \lambda u(O)+ H_i(O, p) \leq 0,$$
and, for any $q \in (\underline q_i, \bar q_i )$, we have
$$ \lambda v(O)+ H_i(O, q) \geq 0.$$
By the continuity of $H_i$, this latter inequality is valid for $q=\underline q_i$ while the former is also valid for $p=\underline q_i$.
Subtracting these inequalities, we obtain a contradiction to \eqref{contraO}.

\smallskip

\noindent
\textit{Case 2. For all $i \in \mathcal{D}$, $-\infty <\underline p_i\leq \bar p_i <\underline q_i \leq \bar q_i \leq +\infty$.}
\smallskip

\noindent
For all $i \in \mathcal{D}$, we consider the functions
\begin{eqnarray*}
p\in\R \mapsto U_i(p):= \lambda u(O)+ H_i(O, p) -\frac{\lambda M}{2},
\end{eqnarray*}
recalling that $M>0$ is defined in~\eqref{contraO}.

On the one hand, from Lemma~\ref{inegGi} (iii), we have $U_i(\bar p_i)< 0$ since $M>0$.
On the other hand, either $\underline q_i <+\infty$ and, again by Lemma~\ref{inegGi} (iii) for supersolutions we have
$$
U_i(\underline q_i) = \lambda v(O) + H_i(O, \underline q_i) + \frac{\lambda M}{2} > 0, 
$$
using the continuity of $H_i$; or $\underline q_i =+\infty$ and, by the coercivity of $H_i$, there exists $\tilde q_i$ large enough such that $U_i(\tilde q_i)>0$.

From these two properties, it follows that there exists $s_i\in (\bar p_i,\underline q_i)$ such that
$U_i (s_i)=0$ and we have
\begin{eqnarray}
  &&  \lambda u(O)+ H_i(O, s_i) = U_i(s_i) +\frac{\lambda M}{2} >0, \nonumber\\
  &&  \lambda v(O)+ H_i(O, s_i) = U_i (s_i)-\frac{\lambda M}{2}< 0.\label{Giv-strict-neg} 
\end{eqnarray}
By the continuity of $H_i$, we can choose $r_i <s_i$, close enough to $s_i$ in order to have 
\begin{equation}\label{Giu-strict-pos}
\lambda u(O)+ H_i(O, r_i)>0.
\end{equation}

Moreover, since $\bar p_i < r_i$ and $s_i < \underline q_i$, we infer that,
for all $i\in  \mathcal{D}$,
\begin{eqnarray}\label{jets-D}
&& r_i \in D^+ u_i(0), \quad s_i \in D^- v_i(0).
\end{eqnarray}

Now, we deal with the indices $i\not\in\mathcal{D}$. We claim that we have $\bar p_i <\underline q_i$.
Indeed, since $\underline p_i=\bar p_i$ and $\underline q_i =\bar q_i $ and since these quantities are finite, $u_i$ and $v_i$ are differentiable at $0$ 
and since $O$ is a strict maximum point of $w=u_i-v_i$, we can use Hopf Lemma, i.e., Lemma~\ref{lem:hopf} in the appendix, which proves the claim.
As a consequence, we can choose as above $\bar p_i < r_i <s_i < \underline q_i$ and \eqref{jets-D} holds.

Finally, gathering all these informations and setting ${\bf r}=(r_1,\cdots, r_N)$ and ${\bf s}=(s_1,\cdots, s_N)$,
we obtain
\begin{eqnarray}\label{diff1}
&& {\bf r} \in D^+ u(O), \quad {\bf s} \in D^- v(O).
\end{eqnarray}
It follows from Proposition~\ref{junct-subdiff} together
with~\eqref{Giu-strict-pos} and~\eqref{Giv-strict-neg}
that necessarily the nonlinear Kirchhoff condition
is activated both for the subsolution $u$ and the supersolution $v$ in~\eqref{eq683bis}
and~\eqref{eq684bis} respectively.
Therefore
\begin{eqnarray*}
&& F(u(O),\mathbf{r})  \leq 0 \leq F(v(O),\mathbf{s}),  
\end{eqnarray*}
which leads to a contradiction by \eqref{hyp-kirch-gene}-(i) since $u(O)\geq v(O)$, $r_i < s_i$ for all $1\leq i \leq N$.
This ends the proof.
\end{proof}

\begin{rema} In the above proof, Hopf Lemma is used in order to have $\bar p_i <\underline q_i$ even if $i \notin \mathcal{D}$.
This is a way to avoid considering different cases, namely whether $\mathcal{D}$ is empty or not, to use \eqref{hyp-kirch-gene}-(i) to obtain the conclusion.
If we strenghten this hypothesis by assuming that, if $r\geq s$ and $p_i\leq q_i$ for all $i$, then
\begin{equation}\label{Strong-hyp-kirch}
F(r,{\bf p})\geq F(s,{\bf q})+ \omega(\min_i(q_i-p_i)),
\end{equation}
for some continuous, increasing function $\omega:[0,+\infty)\to [0,+\infty)$ such that $\omega(0)=0$, then we can argue in a different way.
Indeed we can change the subsolution $u$ in $\hat u(x):=u(x)+\delta x$ for some small $\delta >0$.
This new function is a subsolution of a perturbed equation (within a $O(\delta)$ term independent of $u$, thanks to assumption~\eqref{Ham}-(ii)),
which satisfies the nonlinear Kirchhoff condition in a strict way, i.e., we have
$$ F(\hat u(O), \mathbf{\hat u_x}(O)) \leq -\omega(\delta)<0.$$
This strict subsolution feature allows to conclude in the following way. First, we can assume w.l.o.g. that $\hat u-v$ has still a maximum at $O$, 
otherwise classical arguments apply. Then, the proof can be done as above: in Case~2, we can choose, for all $1\leq i\leq N$, $r_i=s_i$ with $s_i$
satisfying \eqref{Giv-strict-neg} and $\bar p_i < s_i < \underline q_i$. It follows that~\eqref{diff1} holds. Because \eqref{Giv-strict-neg} holds,
the viscosity inequalities---both for the subsolution and the supersolution---reduce to the nonlinear Kirchhoff condition thanks to
Proposition~\ref{junct-subdiff}. Hence we have
\begin{eqnarray*}
&&  F(\hat u(O),\mathbf{s}) \leq -\omega(\delta) \quad\hbox{and}\quad 
F(v(O),\mathbf{s}) \geq 0,
\end{eqnarray*}
which leads to a contradiction since $\hat u(O)=u(O)\geq v(O)$.
\end{rema}

\section{Well-posedness, regularity and and various extensions}\label{vre}

\subsection{Existence for the Kirchhoff-type problem on a junction.}
\label{sec:exis-junct}

The aim of this section is to provide an {\em existence result} for Problem~\eqref{pde-junct}. To do so,
we are going to use the classical Perron's method (cf. Ishii~\cite{ishii87}, see also \cite{cil92}) which requires a stronger assumption than 
\eqref{hyp-kirch-gene}. 

\begin{teo}\label{exist} Under the assumptions of Theorem~\ref{teo:comp} and, if $F$ satisfies in addition: 
there exists $1\leq i_0 \leq N$ such that, for all $r\in \R$ and $\{p_j\}_{j\not= i_0}$, 
\begin{equation}\label{propi0}
\lim_{p_{i_0}\to +\infty} F(r,p_1, \cdots , p_{i_0}, \cdots, p_N)=-\infty,
\end{equation}
then there exists a unique, bounded continuous solution to Problem~\eqref{pde-junct}.
\end{teo}

\begin{proof}
In order to apply the classical Perron's method, the key point is to build sub and supersolutions of the problem. We are going to do it in the following way: we are looking for a supersolution $u^+$ and a subsolution $u^-$ of the form 
$$u^+(x)= K_1-K_2 x \quad \hbox{and}\quad u^-(x)= -K_1+K_2 x,$$
for some constants $K_1,K_2>0$.

In the sequel, we denote by $\1_{\mathbf{\{i\}}}=(0,\cdots, 1, \cdots,0)$ where the $1$ is at the $i^{th}$ coordinates. For $u^+$, using \eqref{hyp-kirch-gene},
we have for $K_2>0$ large enough
$$ F(u^+(O), \mathbf{u_x^+(O)})= F(K_1,-K_2 \1)\geq F(0,-K_2\1_{\mathbf{\{1\}}} )\geq 0.$$
Then we choose $K_1$ large enough in order that we have both
$$
 \lambda u^+(x) - a_i(x)u_{x_i x_i}^+(x) + H_i(x, u_{x_i}^+(x))\geq 0 \quad \hbox{in }E_i
$$ 
(this is possible since $\lambda >0$, $u_{x_i x_i}^+(x)=0$ and $H_i(x, u_{x_i}^+(x))=H_i(x,-K_2)$ is bounded once $K_2$ is fixed) and
$$ u^+(\ver_i) \geq h_i\quad \hbox{for all }i \in \VVb.$$
And similar arguments show that $u^-$ is a subsolution for $K_2$ large enough and then $K_1$ large enough compared to $K_2$
but this time we have to use $\1_{\mathbf{\{i_0\}}}$ instead of $\1_{\mathbf{\{1\}}}$ and Assumption~\eqref{propi0} .

Finally the inequalities $u^-(\ver_i) \leq h_i \leq u^+(\ver_i)$ for all $i\in \VVb$ and $u^-(O) \leq u^+(O)$ also imply that $u^- \leq u^+$ on $\Gamma$
since both $u^- $ and $u^+$ are affine on each $E_i$.

With all these properties, Perron's method applies readily with junction and boundary conditions being understood in the viscosity sense. 
Therefore there exists a discontinuous solution $u$ of Problem~\eqref{pde-junct}. But a straightforward use of Theorem~\ref{teo:comp} implies that $u$ is continuous in $\Gamma\setminus \VVb$ and can be extended as a continuous function on $\Gamma$.
\end{proof}

\subsection{Regularity.}
\label{sec:reg}
The one dimensional structure of the problem allows us to get the following

\begin{prop}\label{prop-reg}
Assume~\eqref{steady} hold, and additionally we assume that for each $1 \leq i \leq N$, either~\eqref{Hcoercive} holds,
or there exists $\underline a_i > 0$ such that $a_i \geq \underline a_i$ in $\bar E_i$. 
Then, all bounded subsolution $u$ to~\eqref{pde-junct} is Lipschitz continuous on the subset 
$\Gamma^\delta := \{ x \in \Gamma : \rho(x, \ver_i) > \delta, \ 1 \leq i \leq N \}$, $\delta >0$,
with a Lipschitz constant which depends on $\delta$ and $\| u\|_\infty$.
%
\end{prop}

\begin{proof}
  For any $y \in \Gamma^\delta$, and $L, K > 0$ to be fixed, we start denoting $\psi_y(x) = L(\rho(x,y) - K\rho(x,y)^2)$, see Lemma~\ref{fct-test-quad}. Assuming $y \in \bar E_i$ for some $i$, and for $R \geq 1$ to be fixed too, we consider the function
$\varphi \in C^1(\Gamma)$ defined by
$$
\varphi(x) = \left \{ \begin{array}{ll} R \psi_y(x) \quad & \mbox{if} \ x \in \bar E_i, \\
R \psi_y(O) +  \psi_O(x)  \quad & \mbox{if} \ x \in E_j,\, j \neq i. \end{array} \right.
$$

Of course, $\varphi$ depends on all the parameters introduced above, but we omit its dependence for simplicity. Notice that when $R = 1$, then $\varphi = \psi$ in Lemma~\ref{fct-test-quad}. This extra parameter allows us to deal with the different natures of the edges and the Kirchhoff condition at once, providing an unified proof.

Consider
$$
\max_{x \in \Gamma \cap \{ \rho(\cdot, y) \leq (4K)^{-1}\}} \{ u(x) - u(y) - \varphi (x)\}.
$$

If we can prove that the maximum is nonnegative for some $L,K$ depending
only on $\delta$ and $\| u\|_\infty$  but not on $y \in \Gamma_\delta$,
then $u$ is Lipschitz continuous with constant $RL$
whenever $\rho(x,y) \leq (4K)^{-1}$,
since
$$
\varphi(x) \leq \left \{ \begin{array}{ll} RL \rho (x,y) \quad & \mbox{if} \ x \in \bar E_i \\
RL \rho (O,y)+L\rho(x,O)\leq RL\rho(x,y) \quad & \mbox{if} \ x \in E_j,\, j \neq i.\end{array} \right.
$$
The Lipschitz bound can then be easily extended to all $x, y \in \Gamma^\delta$. 

So, from now on, we argue by contradiction, assuming that the maximum is positive.
In particular, $\bar x\not= y$ and $\psi$ is differentiable at $\bar x$
by Lemma~\ref{fct-test-quad}-$(i)$.

Thanks again to Lemma~\ref{fct-test-quad} (with $w=u-u(y)$), if $L> 16\, \osc (u) K/3$, then
the above maximum is achieved at $\bar x\in \{ \rho(\cdot, y) < (4K)^{-1}\}$.
In particular, if we suppose in addition that $(4K)^{-1}<\delta$, then
$\rho(\bar x, y)<\delta$ and $\bar x\not\in  \VVb$.

We then distinguish two cases depending whether $\bar x\in E_j$ for some $j$, or $\bar x=O$.
Suppose first that  $\bar x\in E_j$ for some $j$.
We can write the viscosity inequality for $u$ at $\bar x$, that is
$$
\lambda u_j(\bar x) - a_j(\bar x) \varphi_{x_j x_j}(\bar x)+ H_j(\bar x, \varphi_{x_j}(\bar x)) \leq 0.
$$

At this point, we split the analysis depending on the type edge $E_j$ is.
If $H_j$ satisfies~\eqref{Hcoercive}, then we obtain
\begin{eqnarray}\label{choixL-298}
(2C_H)^{-1} L \leq C_H - \lambda u_j(\bar x),
\end{eqnarray}
where we have used~\eqref{gradpsi} and
$R \geq 1$ if $j=i$. Hence, we arrive at a contradiction by taking $L = 2C_H(C_H + \lambda \| u \|_\infty) + 1$.
On the other hand, if $a_j \geq \underline a_j > 0$, using~\eqref{Ham} we have
\begin{align*}
2\underline a_i KL & \leq  C_H L+ |H_i(O, 0)| + C_H\ell_i - \lambda u_i(\bar x) \quad \mbox{if} \ j = i, \\
2\underline a_j R KL & \leq  C_H R L+ |H_j(O, 0)| + C_H\ell_j - \lambda u_j(\bar x) \quad \mbox{if} \ j \neq i,
\end{align*}
from which we also arrive at a contradiction by taking $K\geq (\underline a_j)^{-1}(C_H+1)$
and $L \geq  \max_j  \{|H_j(O, 0)| + C_H\ell_j\}+ \lambda \| u \|_\infty + 1$.
No further restriction on $R$ is needed in this case. 



Hence, the only alternative remaining would be $\bar x = O$ (hence $y\not=O$). Since the maximum is achieved at
$\bar x = O$, we can use Proposition~\ref{reform-junct} for subsolutions.
For each $j \in \mathcal D$,
by the choice of $K, L$ in~\eqref{choixL-298} and since $R \geq 1$, we have 
$$
\lambda u(O) + H_j(O, \varphi_{x_j}(O)) > 0,
$$
from which we necessarily have the nonlinear Kirchhoff condition at $O$, 
\begin{eqnarray}\label{kirchF}
F(u(O), \mathbf{\varphib_{x}(O)} ) \leq 0.
\end{eqnarray}
But, from~\eqref{deriv-varphi}, we get
$$
\varphi_{x_i}(O)= RL(-1+2Ky_i) \quad \text{and} \quad \varphi_{x_j}(O)=L, \ j\not= i.
$$

Recalling that  $\bar x=O\in \{ \rho(\cdot, y) < (4K)^{-1}\}$, we infer
$\varphi_{x_i}(O)\leq -RL/2$.
Therefore, from~\eqref{hyp-kirch-gene}-(ii),
taking $R$ large enough (in terms of $F$ and $||u||_\infty$), we
reach a contradiction in~\eqref{kirchF}. This ends the proof.
\end{proof}

\begin{rema}
A standard bootstrap argument provides $C^{2, 1}$ estimates for \textsl{continuous viscosity solutions} of the problem, on $\bar E_i\cap \Gamma^\delta$, for edges $E_i$ satisfying the uniformly elliptic condition $a_i \geq \underline a_i > 0$ on $\bar E_i$.
\end{rema}

\subsection{Convergence of the vanishing viscosity method}
\label{sec:vanishing-visco}

The result is the
\begin{prop}Under the assumptions of Theorem~\ref{teo:comp}, if $\ue$ is a solution of
\begin{eqnarray*}
&& \left\{\begin{array}{ll}
\lambda \ue  - (\e+a_i(x)) \ue_{x x} + H_i(x,\ue_{x})=0, & x\in E_i, 1\leq i\leq N, \\[2mm] 
F(\ue(O), \ue_{x_1}(O), \cdots ,  \ue_{x_N}(O))=0,\\[2mm] 
\ue(\ver_i)= h_i, & 1\leq i\leq N, 
\end{array}\right.
\end{eqnarray*}
then $\ue \to u$ uniformly on each compact subset of $\Gamma\setminus \VVb$ as $\e\to 0$, where $u$ is the unique solution of Problem~\eqref{pde-junct}.
\end{prop}

We skip the proof since it is a routine adaptation of~\cite[Proposition 4.3 \& Theorem 4.4]{blt24}.
We just point out that the strange form of the convergence we have
for $\ue$ is due to the possible boundary layer we may have at each $\ver_i$.
Indeed, for $\ue$, the Dirichlet boundary condition is satisfied in
a classical sense, i.e., we really have $\ue(\ver_i)= h_i$, while we may have a loss of
boundary condition for $u$,  i.e., we may have $u(\ver_i)< h_i$, see Remark~\ref{loss-bd}.
Hence a uniform convergence on the whole $\Gamma$ is wrong in general.

\subsection{The network case}
\label{sec:network}
Now we consider the case of a finite general network $\Gamma$.
We will use the notations and definitions of~\cite[Section 7.1]{blt24}, that
we recall briefly for convenience.
A network $\Gamma$ is the union of a finite collection of vertices $\VV=\{ \ver_1 , \cdots , \ver_N\}\subset \R^d$
(labeled once for all) and of edges $\EE$.
Each edge  $E\in \EE$ is a line segment in $\R^d$ such that there exists exactly two vertices $\ver_i, \ver_j \in \VV$,
with $i<j$, $\bar E\cap \VV=\{\ver_i , \ver_j\}$, and
\begin{eqnarray}\label{edge345}
E = \{ \ver_i + t(\ver_j - \ver_i)/\ell_E : t \in J_E \},
\end{eqnarray}
where $\ell_E = |\ver_i - \ver_j|$ is the length of the edge and $J_E = (0, \ell_E)$.
The choice $i<j$ in~\eqref{edge345} induces a parametrization of the closure $\bar E$ in $\R^d$ of the edge $E$, that is
$$
\gamma_E: \bar J_E \to \R^N; \quad \gamma_E(t) = \ver_i + t(\ver_j - \ver_i)/\ell_E.
$$
We say that $E \in \EE$ is incident to  $\ver \in \VV$ if $\ver \in \bar E$, and in that case we write $E \in \mathrm{Inc}(\ver)$.
Note that each edge $E \in \EE$ is incident to exactly  two vertices in $\VV$.
We assume that, for each $\ver \in \VV$, there exists $E \in \EE$ with $E \in \mathrm{Inc}(\ver)$.
We write $\VV= \VVi\cup \VVb$, where $\VVi$ is the subset of {\em interior vertices} $\ver$ for which
$N_{\ver}:=\text{card}(\mathrm{Inc}(\ver))> 1$ and $\VVb$ is the subset of {\em boundary vertices} $\ver$ for which
$N_{\ver}$ is exactly one.
The network $\Gamma$ is therefore the subset
$$
\Gamma = \bigcup_{E \in \EE} \bar E.
$$
Finally, we assume that the network is connected in the usual sense.

Using the previous parametrization,
for each $x \in \Gamma$, there exists $E \in \EE$ such that $x \in \bar E$, and in that case we have the existence of a unique $x_E \in \bar J_E$ such that $\gamma_E(x_E) = x$. Thus, for a function $u: \Gamma \to \R$, we write $u_E: \bar J_E \to \R$ such that $u_E(x_E) = u(x)$, that is, the parametrization of $u$ on $\bar E$. Function spaces in Section~\ref{sec:fct-space} can be readily adapted to general networks. Thus, if $u \in C^1(\Gamma)$, we denote
$$
u_x(x) = u_E'(x_E), \quad x \in E, \, E \in \EE,
$$
and we extend it to $\ver \in \VV$ with $E \in \mathrm{Inc}(\ver)$ as
$$
u_{x_E}(\ver) = \lim_{x \to \ver, x \in E} u_x(x),
$$
and analogously for second derivatives.

We define the inward derivative of $u\in C^1(\Gamma)$ at $\ver \in \VV$ as \footnote{Notice that $u_{x_E}(\ver)$ and $\partial_E u(\ver)$ may differ on a minus sign.}
$$
\partial_E u(\ver) = \lim_{x \to \ver, x \in E} \frac{u(x) - u(\ver)}{\rho(x, \ver)}.
$$ 

Now we are in position to present our problem.
For each $E \in \EE$, we consider continuous functions $a_E : \bar E \to [0,+\infty)$, $H_E: \bar E \times \R \to \R$ as part of the equation on each edge $E$. 

For the Kirchhoff vertex condition,
given $\ver \in \VVi$, we list the incident edges as $\{ E_i \}_{i=1}^{N_\ver}$
(this has nothing to do with the original labeling of the vertices).
Then, for $u \in C^1(\Gamma)$, we denote
$$
\boldsymbol{\partial u}(\ver)  = (\partial_{E_1} u(\ver), ..., \partial_{E_{N_{\ver}}} u(\ver)) \in \R^{N_\ver},
$$
and we consider continuous functions $F^{\ver}: \R \times \R^{N_\ver} \to \R$.

For $\ver \in \VVb$, we have $N_\ver = 1$, and we consider a set of real numbers $\{ h_\ver \}_{\ver \in Vb}$ for the Dirichlet boundary conditions.

The problem on a network we have in mind is therefore the following
\begin{equation}\label{eqnetwork}
\left \{ \begin{array}{rll} \lambda u - a_E(x) u_{xx} + H_E(x, u_x) & = 0 \quad & \mbox{on} \ E, \ E \in \EE, \\
F^\ver (u(\ver), \boldsymbol{\partial u}(\ver)) & = 0 \quad & \ver \in \VVi, \\
u(\ver) & = h_\ver \quad & \ver \in \VVb. \end{array} \right .
\end{equation}

Now we present the notion of solution we consider for networks. As before, for each $E \in \EE$ we write
$$
G_E(r, p, X, x) = \lambda r - a_E(x) X + H_E(x, p).
$$
The definition of viscosity solution for problem~\ref{eqnetwork} can be readily adapted from the one given in Definition~\ref{defi-visco}. The Kirchhoff condition on each interior vertex $\ver \in \VVi$ is stated with respect to all its incident edges and the corresponding flux function $F^\ver$. 

In the same way, we are going to consider the set of steady assumptions~\eqref{steady} written in the context of general networks. Namely, degenerate ellipticity assumptions for $a_E$ on each $E \in \EE$ (c.f.~\eqref{hyp-diff}), Lipschitz continuity and/or coercivity for $H_E$ on each $E \in \EE$ (c.f.~\eqref{Ham}), monotonicity and/or coercivity assumptions on $F^\ver$ for each $\ver \in \VVi$ (c.f.~\eqref{hyp-kirch-gene}). The degeneracy nature of the edges~\eqref{ens-deg} must be stated on each vertex, that is, given $\ver \in \VVi$ and taking into account the labeling of the incident edges to $\ver$, we denote
$$
\mathcal D_\ver = \{ 1 \leq i \leq N_{\ver} : a_{E_i}(\ver) = 0, \ E_i \in \mathrm{Inc}(\ver) \}.
$$

\begin{teo}\label{teo:comp-net}
Assume $\Gamma$ is a general network with edges $\EE$ and vertices $\VV = \VVi \cup \VVb$, such that $H_E$ satisfies~\eqref{Ham}, $a_E$ satisfies~\eqref{hyp-diff} for each $E \in \EE$ (with $H_i, a_i$ replaced by $H_E, a_E$).

Assume that for each $\ver \in \VVi$, $F^\ver$ satisfies~\eqref{hyp-kirch-gene} (with $F: \R \times \R^N \to \R$ replaced by $F^\ver: \R \times \R^{N_{\ver}} \to \R$).

Assume that for each $\ver \in \VV$, and each $E \in \mathrm{Inc}(\ver)$, either $H_E$ satisfies~\eqref{Hcoercive} in an $\bar E$-neighborhood of $\ver$, or $a_E$ satisfies $a_E(\ver) > 0$.

Let $\lambda > 0$ and $\{ h_\ver \}_{\ver \in \VVb} \subset \R$ be given.

Then, for each $u \in USC(\Gamma), v \in LSC(\Gamma)$ bounded viscosity sub and supersolution to~\eqref{eqnetwork} respectively, we have $u \leq v$ on $\Gamma \setminus \VVb$. Moreover, if $\tilde u$ is the upper-semicontinuous extension of $u$ at $\ver \in \VVb$ from $E \in \mathrm{Inc}(\ver)$, then $\tilde u \leq v$ on $\Gamma$. 
\end{teo}

This result follows the same ideas of the proof of Theorem~\ref{teo:comp} once we are able to localize around each vertex $\ver \in \VVi$ in the arguments involving $O$ there. First, the analog to Propositions~\ref{reform-junct},~\ref{junct-subdiff},~\ref{PropDir} and Lemma~\ref{inegGi} hold exactly as before since the proofs are local around the vertices. Then, we run the contradiction argument, assuming $M := \sup_{\Gamma} \{ u - v \} > 0$ (where $u$ means $\tilde u$). Since $\Gamma$ is compact, the supremum is attained, and the proof follows standard arguments if the maximum point  is in $\Gamma \setminus \VVi$. Thus, if $u(x_0) - v(x_0) = M$ only for points $x_0 \in \VVi$, we localize around such a point $x_0$, by replacing $u$ by $u - \epsilon \rho^2(x, x_0)$, which, in view of the assumptions, satisfy the same equation as $u$ with a small error term $O(\epsilon)$. Now, $x_0$ is the unique maximum point of $u - \epsilon \rho(x, x_0) - v$ and the proof follows the same lines as above. We omit the details.

\smallskip

Regarding the existence of solutions, we use again
Perron's method. We first state the network version of condition~\eqref{propi0}: for all $\ver \in \VVi$, for each $r \in \R$, $1 \leq i \leq N_\ver$ and all $\{ p_j \}_{j \neq i}$, we have
\begin{equation}\label{Fcoercive-}
\lim_{p_i \to +\infty} F^\ver (r, p_1,...,p_i,...,p_{N_\ver}) = -\infty.
\end{equation}

We assume this condition together with the assumptions of Theorem~\ref{teo:comp-net}.

Since $\VVi$ is finite, from~\eqref{hyp-kirch-gene}-(ii) -- in the network version --, there exists $B > 0$ such that, for all $\ver \in \VVi$, we have 
$$
F^\ver(0, -B, 0,..., 0) > 0.
$$

Thanks to the monotonicity of each $F^\ver$, it follows
\begin{eqnarray}\label{Fnupos}
F^\ver(0, -B {\bf 1}_{N_\ver} ) > 0,
\end{eqnarray}
where ${\bf 1}_{N_\ver} = (1, ..., 1) \in \R^{N_\ver}$. In the same way, enlarging $B$ if necessary and thanks to~\eqref{Fcoercive-}, we will have for each $\ver \in \VVi$ that
$$
F^\ver(0, B {\bf 1}_{N_\ver}) < 0.
$$

Now, consider a function $\theta \in C^2([0,1])$ such that $\theta(x) = -Bx$ for $x \in [0,1/8]$, and $\theta(x) = -B(1 - x)$ for $x \in [7/8, 1]$. We have the existence of a constant $C_B > 0$ such that $\| \theta \|_{C^2([0,1])} \leq C_B$.

Now, let $\Theta \in C^2(\Gamma)$ given by $\Theta(x) = \theta(x_E/\ell_E)$ for $x \in \bar E, E \in \EE$. Then, for a constant $A^+ > 0$ large in terms of $\lambda, C_a, C_H, \min_E \{ \ell_E \}$ and $B$, we have 
$$
G_E(A^+ + \Theta(x), \Theta_x(x), \Theta_{xx}(x), x) \geq 0 \quad \mbox{for $x \in E, \ E \in \EE$}.
$$
Thus $\Theta^+ := A^+ + \Theta$ is a viscosity supersolution in each edge $E \in \EE$.

Moreover, thanks to~\eqref{Fnupos} and to the monotonicity of $F^\ver$ again, for all $\ver \in \VVi$, we have
$$
F^\ver(\Theta^+(\ver), \boldsymbol{\partial \Theta^+}(\ver)) \geq 0.
$$
Finally, enlarging $A^+$ in terms of $\max_{\ver \in \VVb} |h_\ver|$, we have $\Theta^+(\ver) \geq h_\ver$ for all $\ver \in \VVb$. Thus, we have constructed a viscosity supersolution to~\eqref{eqnetwork}. 

Notice that by the same arguments, the function $\Theta^-(x) = -\Theta^+(x)$ is a viscosity subsolution to the same problem, and enlarging $A^+$ is necessary, we have $\Theta^- \leq \Theta^+$ on $\Gamma$. Then, Perron's method can be readily adapted to the context of general (finite) networks, from which we conclude the result.

Following these ideas, we have
\begin{teo}\label{teo:Perron-net}
  Under the assumptions of Theorem~\ref{teo:comp-net} and~\eqref{Fcoercive-}, there exists a unique viscosity solution $u \in C(\Gamma)$ for~\eqref{eqnetwork}. Moreover, with a straightforward adaptation of the assumptions of Proposition~\ref{prop-reg},
  the solution is locally Lipschitz continuous on $\Gamma \setminus \VVb$.
\end{teo}

\subsection{Easy and more tricky possible extensions}\label{easytricky}

We list here some possible extensions: the first ones are routine adaptations of the methods we use here, the others need to be double-checked.

\subsubsection{More general equations}
We have chosen to present our results with a second-order term which is just
a linear term on each edge, namely $-a_i(x)u_{x_i x_i}$
but there is no difficulty to handle a fully nonlinear
term like $A_i (x,u_{x_i},u_{x_i x_i})$ under the
following natural assumptions:
\begin{itemize}
\item[(i)] There exists a constant $C$ such that, for any $1\leq i \leq N$, $x\in \bar E_i$, $p,q,X,Y\in \R$
$$ \vert A_i (x,p,X)-A_i(x,q,Y)\vert \leq C(|p-q|+|X-Y|).$$

\item[(ii)] There exists a modulus of continuity $\omega: [0,+\infty)\to [0,+\infty)$ such that, for any $1\leq i \leq N$,
$x,y\in \bar E_i$, $p,X,Y\in \R$ such that
$$ \left(\begin{array}{cc}X & 0 \\0 & -Y\end{array}\right)\leq \frac{1}{\e}\left(\begin{array}{cc}1 & -1 \\ -1 & 1\end{array}\right),$$
for some $\e >0$, then
$$  A_i (x,p,X)-A_i(x,q,Y)\leq \omega\left(|x-y|(1+|p|) + \frac{|x-y|^2}{\e}\right).$$

\item[(iii)] For any $1\leq i \leq N$, either $A_i (O,p,X)=0$ for any $p,X\in \R$, or there exists $\eta >0$ such that, for any $x$ in an $E_i$-neighborhood of 
$O$ and for any $p\in \R$, we have
$$ A_i (x,p,X)-A_i(x,p,Y) \leq \eta(X-Y) \quad \hbox{if  }X\geq Y\; .$$

\end{itemize}

Assumptions~(i)-(ii) are basic assumptions in classical comparison results (see e.g. Theorem 3.3 in~\cite{cil92}), while (iii) is the analogue of the dichotomy either ``$a_i(O)>0$ or $a_i(O)=0$''.

\subsubsection{More general boundary conditions}
The Dirichlet boundary conditions on $\VVb$ can easily be changed to Neumann boundary conditions.
Indeed, the exterior boundary conditions on  $\VVb$ appear as a particular case of a Kirchhoff condition
at a vertex, which is connected to only one edge. As a consequence, 
the comparison  inequality holds on the whole $\Gamma$.

The case of unbounded edges $E$, for example, $E=(0,+\infty)$, can be treated by usual localization arguments.

Finally, it is also possible to mix the above boundary conditions in the same network without major difficulties.

\subsubsection{Further (more tricky) extensions and discussions}
We present here some generalizations that we would like to address
in future works.
\smallskip

As explained in Section~\ref{sec:hamilt}, we have considered
Hamiltonians satisfying the standard optimal control assumptions. But
Several results can be readily applied to more general Hamiltonians.
\smallskip

This article just considers the case of one dimensional networks but what about the case of multi dimensional networks? Typically
$\Gamma \times \R$. This introduces a new variable---say $z$---and we are not anymore in the framework of Lions and Souganidis \cite{ls16,ls17}. 
For example, we may have in $E_i\times \R$ an equation like
$$
 \lambda u(x,z) - a^1_i(x,z)u_{xx}(x,z)- a^2_i(x,z)u_{zz}(x,z) + H_i(x, z,u_x(x,z),u_z (x,z))\geq 0,
$$ 
with $a^1_i(x,z)$ playing the same role as $a_i(x)$ above. Even with rather restrictive assumptions, we have to handle one way or the other the new ``tangential terms'' $u_z (x,z)$, $a^2_i(x,z)u_{zz}(x,z)$.
This case is particularly relevant since it encompasses
evolution problems with the form
$$
 u_t (x,t) - a_i(x,t)u_{xx}(x,t) + H_i(x,t, u_x (x,t))= 0 \quad \hbox{in }E_i \times (0,T)\; .
$$ 
with $0 < T \leq +\infty$.

Two ideas can be tried to solve such problems:\\
(i) The ``tangential regularization'' used in \cite{bc24} can perhaps handle some of these cases but maybe not all of them because it requires
strong assumptions to be performed.\\
(ii) The ``twin blow-up method'' of Forcadel, Imbert and Monneau~\cite{fim25} may be a little bit more promising, but it does not solve all the difficulties by itself.
\smallskip

Another natural question is related to the extension of these results for infinite networks,
meaning there are an infinite number of interior vertices and edges. The typical
case is a standard lattice in $\R^d$, for instance when $\VVi =\Z^d$, $\VVb =\emptyset$ and
$\EE =\{ \ver +(0, 1) {\mathrm{\mathbf{e}}_{i}} :  1\leq i\leq d, \ver\in \VVi \}$,
where $\{{\mathrm{\mathbf{e}}_{i}}\}_{i=1}^{d}$ is the standard orthonormal basis of $\R^d$. Even in simple cases, the non-finiteness of $\VVi$ put difficulties on the application of the the arguments of the finite case, mostly related to localization procedures to deal with the lack of compactness of $\Gamma$. Thus, additional assumptions must be considered, and new ideas are required.

\appendix

\section{}

The test-function $\psi$ defined by~\eqref{psi-gene}
is extensively used in the paper for various parameters $y, L, K$.
For convenience, we state the basic useful properties in a lemma,
the (easy) proof of which is left to the reader.

\begin{lema}\label{fct-test-quad}
Let $w\in USC(\Gamma)$ be bounded. For all $y\in \Gamma$, $L,K>0$, we define
\begin{eqnarray}\label{psi-gene}
\psi(x)=\psi_{y,L,K}(x):= L\left( \rho(x,y) - K \rho(x,y)^2\right),
\end{eqnarray}  
where $\rho$ is the geodesic distance defined in~\eqref{def-geod}.

\begin{itemize}
\item[(i)] $\psi\in C^2(\Gamma\setminus \{y\})$ and, if $y\in \bar E_i$ for some $i$, then
\begin{eqnarray}\label{deriv-varphi}
 && \begin{array}{ll}
    \psi_{x_i}(x)= L\left( \frac{x_i -y_i}{|x_i -y_i|}-2K (x_i -y_i)\right) & \text{if $x\in \bar E_i$, $x\not= y$,}\\
    \psi_{x_j}(x)= L\left( 1- 2K (x_i +y_i)\right) & \text{if $x\in \bar E_j$, $j\not= i$,}\\
    \psi_{x_j x_j}(x) = -2LK \leq 0 & \text{for all $x\not= y$, $1\leq j\leq N$.}
  \end{array}
\end{eqnarray}

\item[(ii)] For every $K$, the maximum
\begin{eqnarray*}
&& \max_{\Gamma \cap \{ \rho(\cdot, y) \leq (4K)^{-1}\}} w-\psi
\end{eqnarray*}
is achieved at some $\bar x \in \Gamma$ such that
\begin{eqnarray}\label{psixbarre}
&& \frac{3L}{4} \rho (\bar x, y) \leq \psi (\bar x) \leq w(\bar x) - w(y) \leq \osc(w)\leq 2 ||w||_\infty, 
\end{eqnarray}  
where $\osc(w) = \sup_\Gamma w - \inf_\Gamma w$. 

If $L > 16\, \osc(w) K /3$, then $\bar x \in \{\rho(\cdot , y) < (4K)^{-1} \}$.
Moreover, if $\bar x\not= y$, then
\begin{eqnarray}\label{gradpsi}
   \frac{L}{2} \leq |\psi_{x_j} (\bar x)| \leq L & \text{for all $j$,}
\end{eqnarray}

\end{itemize}

\end{lema}

\begin{lema}\label{lem-max123}
Let $a>0$ and $w\in USC ([0,a])$, define
$$\displaystyle
-\infty \leq \underline p =  \mathop{\rm lim\,inf}_{x\to 0^+}\frac{w(x)-w(0)}{x} \ \leq \ 
\bar p= \mathop{\rm lim\,sup}_{x\to 0^+}\frac{w(x)-w(0)}{x} \leq +\infty.$$
and, for $p, X\in \R$,
$$\chi (x)= w(x)-w(0)-px+\frac{1}{2}X x^2.$$
Then
\begin{itemize}
\item[(i)] for all $p>\bar p$ and $X\in\R$, the function $\chi$
  has a strict local maximum at $0$ and $(p,-X)\in J^{2,+}_{[0,a]}w(0)$.
\item[(ii)] for all $\underline p < p <\bar p$,   
there exists sequences $0 < x_k < b_k \to 0$ such that $x_k$ is a maximum
point of $\chi$ in $(0, b_k)$.
\end{itemize}
\end{lema}

\begin{proof}
(i) The statement only makes sense when $\bar p<+\infty$. In this case,
choose any $X\in\R$, $p> \bar p$ and $\tilde p \in (\bar p, p)$ (it is possible
to choose $\tilde p = \bar p$ if $\bar p>-\infty$).
By definition of $\bar p$,
$\frac{w(x)-w(0)}{x}\leq \tilde p +o(1)$, where $o(1)\to 0$ as $x\to 0^+$.
It follows that
\begin{eqnarray*}
\psi(x):= w(x)-w(0) -px +\frac{1}{2}Xx^2 \leq \left( \tilde p - p +\frac{1}{2}Xx +o(1)\right) x.
\end{eqnarray*}
But $\tilde p - p +\frac{1}{2}Xx +o(1) <0$ if $0<x<\delta$ where $\delta=\delta(w,p,X)$ is sufficiently
small. Hence $\psi(x)<0$ on $(0,\delta)$ with $\psi(0)=0$,
which proves that $\psi$ has a strict maximum at $0$ on $[0,\delta]$
and  $(p,-X)\in J^{2,+}_{[0,a]}w(0)$.

We turn to the proof of (ii).
Let $X\in\R$, $p \in (\underline p, \bar p)$ and
choose $p_1, p_2\in \R$ such that $\underline p < p_1 < p < p_2 <\bar p$
(the introduction of $p_1, p_2$ is convenient to cover the case when $\underline p$ or $\bar p$
are not finite).
By definition of $\underline p, \bar p$, there exists sequences $0 < a_k < b_k$ with $b_k\searrow 0$
as $k\to \infty$, such that
\begin{equation*}\label{ineg498}
w(b_k) - w(0) \leq (p_1 + k^{-1}) b_k , \quad w(a_k) - w(0) \geq (p_2 - k^{-1})a_k.
\end{equation*}
It follows that
\begin{eqnarray*}
 \chi(b_k)= \left( \frac{w(b_k)-w(0)}{b_k}- p +\frac{1}{2} X b_k\right) b_k
  \leq \left(p_1 + \frac{1}{k} + \frac{1}{2} X b_k -p\right) b_k<0
\end{eqnarray*}
as soon as $k$ is large enough in order that $p> p_1 + k^{-1} + 2^{-1}X b_k$.
Similarly $\chi(a_k)>0$ when $p < p_2 - k^{-1} - 2^{-1}X b_k$.
Since $\chi\in USC [0, 1]$ satisfies  $\chi(0)=0$, we
conclude for the existence of a maximum point $x_k\in (0,b_k)$.
\end{proof}

\begin{lema}[Hopf Lemma]\label{lem:hopf}
Let $\ell >0$ and $w\in USC([0,\ell])$ be a viscosity subsolution of
\begin{eqnarray}\label{sub-pde}
\lambda w - a(x) w_{x x} -L |w_{x}|\leq 0, \quad x\in (0,\ell ),
\end{eqnarray}
where $a$ satisfies~\eqref{hyp-diff} with $a(0)>0$ and $\lambda, L>0$.
If $w$ has a positive strict local maximum point at $0$ and $w$ is differentiable at $0$,
then $- w_x(0) >0$.
\end{lema}

\begin{proof}
See~\cite[p.330]{evans98}. 
The adaptation is straightforward since the weak maximum principle
holds for viscosity solutions of~\eqref{sub-pde},
the equation is elliptic near $0$, and $w$ is assumed to be differentiable at $0$.
\end{proof}


\end{document}